\documentclass[a4paper,11pt]{amsart}
\usepackage{rotating}
\usepackage{pdflscape}
\usepackage{amsmath}
\usepackage{amssymb}
\usepackage{color}
\usepackage{enumerate}
\usepackage{amsthm}
\usepackage{listings}
\usepackage[all]{xy}
\usepackage{longtable}
\usepackage{graphicx}
\usepackage{hyperref}
\usepackage{mathrsfs}
\usepackage[style=alphabetic,backend=bibtex]{biblatex}
\theoremstyle{definition}
\newtheorem{theorem}{Theorem}[section]
\newtheorem{proposition}[theorem]{Proposition}
\newtheorem{lemma}[theorem]{Lemma}
\newtheorem{corollary}[theorem]{Corollary}
\newtheorem{definition}[theorem]{Definition}
\newtheorem{remark}[theorem]{Remark}
\newtheorem{example}[theorem]{Example}

\newcommand{\reg}{\mathrm{reg}}
\newcommand{\sing}{\mathrm{sing}}

\newcommand{\Kosz}{\operatorname{Kosz}}
\newcommand{\one}{\mathbf{1}}
\newcommand{\Hom}{\operatorname{Hom}}

\newcommand{\Span}{\operatorname{span}}
\newcommand{\rank}{\operatorname{rank}}
\newcommand{\codim}{\operatorname{codim}}

\newcommand{\D}{\operatorname{d}\!}
\newcommand{\CC}{\mathbb{C}}

\newcommand{\ZZ}{\mathbb{Z}}
\newcommand{\NN}{\mathbb{N}}

\newcommand{\PP}{\mathbb{P}}
\newcommand{\OO}{\mathcal{O}}

\newcommand{\Grass}{\operatorname{Grass}}

\newcommand{\Stief}{\operatorname{Stief}}

\newcommand{\GL}{\operatorname{GL}}

\newcommand{\diag}{\operatorname{diag}}
\newcommand{\Stab}{\operatorname{Stab}}

\title{On the topology of determinantal links}

\author{Matthias Zach}

\begin{document}

\begin{abstract}
  We study the cohomology of the generic determinantal 
  varieties $M_{m,n}^s = \{ \varphi \in \CC^{m\times n} : \rank \varphi <s \}$,
  their polar multiplicities, 
  their sections $D_k \cap M_{m,n}^s$ by generic hyperplanes $D_k$ 
  of various codimension $k$, and the real and complex links 
  of the spaces $(D_k\cap M_{m,n}^s,0)$. 
  Such complex links were shown to provide the basic building blocks in a 
  bouquet decomposition for the (determinantal) smoothings of 
  smoothable isolated determinantal singularities. 
  The detailed vanishing topology of such singularities was still not fully understood 
  beyond isolated complete intersections and a few further special cases. 
  Our results now allow to compute all distinct Betti numbers of any determinantal smoothing.
\end{abstract}

\maketitle

\tableofcontents

\newpage 

\section{Introduction and results} 

The motivation for this paper is to understand the vanishing (co-)homology of 
isolated \textit{determinantal singularities} $(X,0) \subset (\CC^p,0)$ 
(abbreviated as IDS in the following) which admit a 
\textit{determinantal smoothing} $M$. Despite their seemingly odd definition
(see below), 
they are quite frequently encountered; for instance any normal surface 
singularity in $(\CC^4,0)$ and any so-called ``space curve'' $(C,0) \subset (\CC^3,0)$ 
is determinantal by virtue of the Hilbert-Burch theorem, see 
e.g. \cite{Wahl16} and \cite{Schaps77}. It was shown in \cite{Zach18BouquetDecomp} 
that sections of the generic determinantal varieties 
$M_{m,n}^s = \{\varphi \in \CC^{m\times n} : \rank \varphi < s \}$ by 
hyperplanes $D_k' \subset \CC^{m\times n}$, which are of codimension $k$ 
and in general position off the origin, 
provide the fundamental building blocks in a bouquet decomposition of the 
determinantal smoothing $M$. 

In this paper we aim to study these building 
blocks 
\[
  \mathcal L^k(M_{m,n}^s,0) \cong M_{m,n}^s \cap D'_{k+1},
\]
which we shall refer to as \textit{complex links of codimension} $k$, 
and discuss the implications for the vanishing topology of various singularities. 
Our main result in this regard is that whenever $k$ is sufficiently big so that 
the hyperplane $D'_{k+1}$ misses the singular locus of $M_{m,n}^s$ and 
$\mathcal L^k(M_{m,n}^s,0)$
is smooth, then the cohomology in degrees $i$ 
\textit{below} the middle degree $d = \dim \mathcal L^k(M_{m,n}^s,0)$ 
is 
\[
  H^i(\mathcal L^k(M_{m,n}^s,0)) \cong H^i(\Grass(s-1,m)),
\]
that is, it is isomorphic to the cohomology of a Grassmannian, 
see (\ref{eqn:CohomologyOfSmoothComplexLinksBelowMiddle}).

In the general case, i.e. for arbitrary codimension $k$, we 
provide a formula (\ref{eqn:FormulaForEulerCharacteristic})
for the Euler characteristic of 
$\mathcal L^k(M_{m,n}^s,0)$ based on the polar methods developed 
by L\^e and Teissier. 
We proceed to study the polar multiplicities $e_{m,n}^{r,k}$ of 
the generic determinantal varieties 
$(M_{m,n}^{r+1},0)$ 
in Section \ref{sec:PolarMultiplicities} and obtain  
a formula (\ref{eqn:FormulaForThePolarMultiplicity}) for these 
numbers as an integral of Segre classes.
It should be noted that this formula does not depend on 
the choice of any generic hyperplane section and it can be 
implemented effectively in a computer algebra system. 
However, we were unable to prove a closed formula for the numbers 
$e_{m,n}^{r,k}$ as a function of $m,n,r$, and $k$;
even though the calculated Tables 
\ref{tab:PolarMultiplicitiesHilbertBurch}, and 
\ref{tab:PolarMultiplicities2byNRank1}
-- \ref{tab:PolarMultiplicities5byNRank1and4}
yield several patterns.

Coming back to the case of smooth links, 
this allows then to also compute the middle Betti number of 
$\mathcal L^k(M_{m,n}^s,0)$ (but unfortunately not the full cohomology group 
with integer coefficients in general, i.e. there could also be torsion). 

\medskip
We would like to point out that many methods developed in this 
article apply much more generally to compute the 
cohomology of the links of higher codimension for spaces 
$(X,0) \subset (\CC^q,0)$ which decompose into orbits 
of a Lie group action $G \times X \to X$. 
The generic determinantal varieties treated here are just one particular, 
accessible example.
\medskip

We continue with the discussion of the 
implications of these results for arbitrary smoothable IDS.
Let 
\[
  A \colon (\CC^p,0) \to (\CC^{m\times n},0) 
\]
be a holomorphic map germ to the space of $m\times n$-matrices. 
We say that $(X,0)$ is \textit{determinantal for} $A$ and of type $(m,n,s)$ 
if we have $(X,0) = (A^{-1}( M_{m,n}^s),0)$ 
and such that $\codim(X,0) = \codim(M_{m,n}^s,0) = (m-s+1)(n-s+1)$. 
In this case we also write $(X,0) = (X_A^s,0)$ to emphasize 
the chosen determinantal structure for $(X,0)$. 

By definition, a \textit{determinantal deformation} is induced 
by an unfolding $\mathbf A(x,t)$ of $A(x)$ on parameters $t = (t_1,\dots,t_k)$ such that 
$A_0 = \mathbf A(-,0)$ is equal to the original map $A$ and 
$A_t= \mathbf A(-,t)$ 
is transverse to $M_{m,n}^s$ in a stratified sense for $t \neq 0\in \CC^k$.
Then $X_t = A_t^{-1}(M_{m,n}^s)$ is a smoothing of $(X,0)$, provided 
that $p$ is strictly smaller than the codimension $c = (m-s-2)(n-s-2)$ of the 
singular locus of $M_{m,n}^s$. In this case we shall speak of 
the \textit{determinantal smoothing}\footnote{Note that in 
  the case of an isolated determinantal hypersurface singularity given by 
  $(X,0) = (\{ \det A = 0 \},0) \subset (\CC^p,0)$ 
a smoothing of the singularity always exists, 
but it can be realized by a determinantal deformation 
only when $p<4$.} of $(X_A^s,0)$ 
and it can be shown that, up to diffeomorphism, 
this smoothing depends only 
on the choice of the determinantal structure for $(X,0)$, but not on the particular 
unfolding of the matrix. We will therefore also write $M_A^s$ for the 
determinantal smoothing $X_t$ of $(X,0)$.

Note that every 
complete intersection singularity $(X,0)\subset (\CC^p,0)$ 
of codimension $k$ is determinantal 
of type $(1,k,1)$ for the matrix $A$ whose entries are given by a regular 
sequence generating the ideal of $(X,0)$. In this case the determinantal deformations 
coincide with the usual deformations of the singularity. We refer 
to \cite{FruehbisKruegerZach21} for the details on determinantal singularities 
and their deformations and further references for the above statements.

The aforementioned bouquet decomposition of the determinantal smoothing 
$M_A^s$ of an IDS $(X,0) = (X_A^s,0)$ from \cite{Zach18BouquetDecomp} reads
\begin{equation}
  M_A^s \cong_{\mathrm{ht}} \mathcal L^{mn-p-1}(M_{m,n}^s,0) \vee 
  \bigvee_{i=1}^\lambda S^{\dim(X_A^s,0)}
  \label{eqn:BouquetDecomposition} 
\end{equation}
where we denote by $S^d$ the unit sphere of real dimension $d$. 
Combining this with our results we find

\begin{theorem}
  Let $M_A^s$ be the determinantal smoothing of a $d$-dimensional smoothable 
  isolated determinantal singularity $(X_A^s,0) \subset (\CC^p,0)$ 
  of type $(m,n,s)$ with $s\leq m\leq n$. Then the truncated cohomology 
  \[
    H^{\leq d}(\Grass(s-1,m)) 
    \subset H^\bullet(M_A^s) 
  \]
  embeds into the cohomology of $M_A^s$ with a quotient concentrated 
  in cohomological degree $d$. 
  Moreover, the left hand side is generated as an algebra 
  by the Segre classes of the vector bundle 
  $E$ on $M_A^s$ whose fiber at a point $x$ is presented by the 
  perturbed matrix 
  \[
    \CC^n \overset{A_t(x)}{\longrightarrow} \CC^m \to E_x \to 0
  \]
  which defines the smoothing $M_A^s$.
  \label{thm:MainConsequence}
\end{theorem}

\begin{landscape}
\begin{table}
  \begin{tabular}{|c|cccccccccc|}
    \hline
    $m \backslash k$ & 
    0 & 1 & 2 & 3 & 4 & 5 & 6 & 7 & 8 & 9 \\
    \hline
    1 &
    1 & 0 & 0 & 0 & 0 & 0 & 0 & 0 & 0 & 0 \\ 
    2 &
    3 & 4 & 3 &0  &0  &0  &0  &0  &0  &0  \\
    3 &
    6 & 16 & 27 & 24 & 10 & 0& 0& 0& 0& 0\\
    4 & 
    10 & 40 & 105 & 176 & 190 & 120 & 35 & 0& 0& 0\\
    5 &
    15 & 80 & 285 & 696 & 1200 & 1440 & 1155 & 560 & 126 &0\\
    6 & 
    21 & 140 & 630 & 2016 & 4760 & 8352 & 10815 & 10080 & 6426 & 2520 \\
    7 & 
    28 & 224 & 1218 & 4816 & 14420 & 33216 & 59143 & 80976 & 83916 & 63840 \\
    8 & 
    36 & 336 & 2142 & 10080 & 36540 & 104112 & 235557 & 424384 & 606564 & 680400 \\
    9 & 
    45 & 480 & 3510 & 19152 & 81480 & 276480 & 758205 & 1691920 & 3077838 & 4551840 \\
    10 & 
    55 & 660 & 5445 & 33792 & 165000 & 649440 & 2091705 & 5563360 & 12278970 & 22518200 \\
    \hline
    \hline

    $m \backslash k$ & 10 & 11 & 12 &13 &14 &15 &16 &17 &18 &19 \\
    \hline
    6 &
    462 & 0 & 0 & 0 & 0 & 0 & 0 & 0 & 0 & 0 \\
    7 & 33726 & 11088 & 1716 & 0& 0& 0& 0& 0& 0 &0 \\
    8 & 587202 & 376992 & 169884 & 48048 & 6435 &0&0&0&0&0\\
    9 & 5430810 & 5155920 & 3809520 & 2114112 & 830115 & 205920 & 24310 &0&0&0 \\
    10 & 34240800 & 42926400 & 43929600 & 36132096 & 23326875 & 11394240 & 3962530 & 875160 & 92378 & 0 \\
    \hline
  \end{tabular}
  \caption{Polar multiplicities $e_{m,m+1}^{m-1,k}$ for the generic determinantal varieties 
  $(M_{m,m+1}^m,0)$ appearing in the context of the Hilbert-Burch theorem}
  \label{tab:PolarMultiplicitiesHilbertBurch}
\end{table}

\end{landscape}

In the case of smoothable isolated Cohen-Macaulay codimension $2$ 
singularities $(X,0) \subset (\CC^p,0)$
we obtain some more specific results which might be of particular interest.
As remarked earlier, these singularities admit a canonical determinantal 
structure by virtue of the Hilbert-Burch theorem (see \cite{Hilbert90}, \cite{Burch68}, or 
\cite{Eisenbud95} for a textbook): 
Let $f_0,\dots,f_m \in \CC\{x_1,\dots,x_p\}$ be a minimal set of generators 
for the ideal $I$ defining $(X,0)$. Then the minimal free resolution of 
$\OO_{X,0} = \CC\{x_1,\dots,x_p\}/I$ over the ring $\OO_p = \CC\{x_1,\dots,x_p\}$
takes the form 
\[
  0 \to \OO_p^m \overset{A^T}{\longrightarrow} \OO_p^{m+1} \overset{f}{\longrightarrow} 
  \OO_p \to \OO_{X,0} \to 0 
\]
for some matrix $A \in \OO^{m\times (m+1)}$ and then 
$(X,0)= (X_A^m,0)$ is determinantal 
of type $(m,m+1,m)$ for the matrix $A$. Moreover, any deformation of 
$(X,0)$ is also automatically determinantal (see \cite{Schaps77}), i.e. 
it arises from a perturbation of $A$. Isolated Cohen-Macaulay 
codimension $2$ singularities exist up to dimension $d = \dim(X,0) \leq 4$ 
and they are smoothable if and only if this inequality is strict.

The generic determinantal varieties appearing in the context of the 
Hilbert-Burch theorem are naturally limited to $M_{m,m+1}^m$ and 
we list their polar multiplicities for values $0\leq m\leq 10$ in Table 
\ref{tab:PolarMultiplicitiesHilbertBurch}. From 
these numbers we can then also compute the Euler characteristic of 
the smooth complex links of their generic hyperplane sections, 
listed in Table \ref{tab:EulerCharacteristicsLinksHilbertBurch}.

\begin{table}
  \centering
  \begin{tabular}{|c|cccccccccc|}
    \hline
    $d \backslash m$ & 
    1 & 2 & 3 & 4 & 5 & 6 & 7 & 8 & 9 & 10 \\
    \hline 
    0 & 1 & 3 & 6 & 10 & 15 & 21 & 28 & 36 & 45 & 55 \\
    1 & 0 & -1 & -10 & -30 & -65 & -119 & -196 & -300 & -435 & -605 \\ 
    2 & 0 & 2 & 17  & 75  & 220 & 511 &  1022 & 1842 & 3075 & 4840 \\
    3 & 0 & 2 & -7 & -101 & -476 & -1505 & -3794 & -9138 & -16077 & -28952 \\
    \hline
  \end{tabular}
  \caption{The Euler characteristic of the smooth complex links 
    $\mathcal L^{m(m+1)-d-3}(M_{m,m+1}^m,0)$ of dimension $d$ of 
  the generic determinantal varieties appearing in the context of the 
Hilbert-Burch theorem}
  \label{tab:EulerCharacteristicsLinksHilbertBurch}
\end{table}

As a consequence of Theorem \ref{thm:MainConsequence} we obtain: 
\begin{corollary}
  Let $(X,0) \subset (\CC^5,0)$ be a Cohen-Macaulay germ of codimension 
  $2$ with an isolated singularity at the origin and $M$ its smoothing. 
  Then $H^2(M) \cong \ZZ$ is free of rank one and generated 
  by the Chern class of the canonical bundle on $M$. 
  \label{cor:SecondCohomologyGroupSmoothingICMC2Threefold}
\end{corollary}

\noindent
This has been conjectured by the author in \cite{Zach18} where 
a special case of Corollary \ref{cor:SecondCohomologyGroupSmoothingICMC2Threefold} 
was obtained for those singularities defined by $2\times 3$-matrices. 
Our results on the Euler characteristic of the complex links also give 
lower bounds for the third Betti number 
\begin{equation}
  b_3 \geq - \chi\left( \mathcal L^{m(m+1)-6}(M_{m,m+1}^m,0) \right) - 2
  \label{eqn:BoundForMiddleBettiNumberThreefolds}
\end{equation}
which can be read off directly from Table \ref{tab:EulerCharacteristicsLinksHilbertBurch} 
depending on the size of the matrix $m$. Unfortunately, we were unable 
to prove a closed algebraic formula for this number as a function of $m$.

While in the case of threefolds the cohomology of the 
complex link prominently 
sticks out in the sense that it entirely covers the nontrivial contributions 
outside the middle degree, the situation is a little more hidden 
in dimensions $d<3$ since then all the non-trivial (reduced) cohomology 
of the smoothing $M$ is concentrated in the middle degree. 
Nevertheless, using Theorem \ref{thm:MainConsequence}, 
a lower bound for the middle Betti number can be read 
off from Table \ref{tab:EulerCharacteristicsLinksHilbertBurch} 
as 
\begin{equation}
  b_2 \geq \chi\left( \mathcal L^{m(m+1)-5}(M_{m,m+1}^m,0) \right) - 2
  \label{eqn:BoundForMiddleBettiNumberSurfaces}
\end{equation}
for smoothings of isolated Cohen-Macaulay codimension $2$ surfaces 
and 
\begin{equation}
  b_1 \geq -\chi\left( \mathcal L^{m(m+1)-4}(M_{m,m+1}^m,0) \right) - 1
  \label{eqn:BoundForMiddleBettiNumberSpaceCurves}
\end{equation}
in the case of space curves in $(\CC^3,0)$. 
The proofs of Theorem \ref{thm:MainConsequence} and Corollary 
\ref{cor:SecondCohomologyGroupSmoothingICMC2Threefold}
will be given in the last section.

\begin{remark}
  Wahl has shown in \cite{Wahl16} that for a normal surface singularity 
  $(X,0) \subset (\CC^4,0)$ which is not Gorenstein, 
  the second Betti number $\mu$ of the 
  smoothing $M$ and the Tjurina number $\tau$ of $(X,0)$ obey an 
  inequality 
  \[
    \mu \geq \tau - 1
  \]
  with equality whenever $(X,0)$ is quasi-homogeneous. In particular, this 
  applies to all generic hyperplane sections $(X,0) = (D \cap M_{m,m+1}^m,0) \subset (D,0)$ 
  for $m>1$ and $D$ of dimension $4$. It follows from the above that 
  $\chi(M) = \tau$ and the first few values can be read off from 
  Table \ref{tab:EulerCharacteristicsLinksHilbertBurch}. Comparing Wahl's results 
  with the bouquet decomposition (\ref{eqn:BouquetDecomposition}) we see 
  that for non-linear quasi-homogeneous singularities, 
  the difference of the numbers $\tau$ and $\lambda$ is determined by the offset 
  \[
    \tau - \lambda = \chi\left( \mathcal L^{m(m+1)-5}(M_{m,m+1}^m,0) \right)
  \]
  depending only on the size of the matrix $m >1$. This is different 
  compared to the case of isolated complete intersections (the case $m=1$) 
  where we have equalities $\tau = \mu = \lambda$; see \cite{Wahl85}.
\end{remark}

\section{Real and complex links of higher codimension} 

\subsection{Definitions}

Let $X \hookrightarrow U \subset \CC^n$ be the closed embedding of an 
equidimensional reduced complex analytic variety of dimension $d$ in some 
open set $U$ of $\CC^n$ and suppose $\{ V^\alpha\}_{\alpha\in A}$ 
is a complex analytic Whitney stratification for $X$. Such stratifications 
always exist; a construction of a 
unique minimal stratification has been described in \cite{LeTeissier81} and 
if not specified further, we will in the following always assume $X$ to be endowed 
with its minimal Whitney stratification. Moreover, we may suppose that 
every stratum is connected for if this was not the case, we could replace 
it by its distinct irreducible components. 

Recall the notions of the \textit{real} and \textit{complex links} of $X$ 
along its strata. These can be defined as follows: Let $V^\alpha$ be 
a stratum of $X$ and $x \in V^\alpha \subset X$ an arbitrary point. 
For simplicity, we may assume $x = 0 \in \CC^n$ to be the origin. Choose a
\textit{normal slice} to $V^\alpha$ through $x$, i.e. a 
submanifold $(N_x,x) \subset (\CC^n,x)$ 
of complementary dimension which meets $V^\alpha$ transversally in $x$.
Let $B_\varepsilon(x)$ be the closed ball of radius $\varepsilon$ centered 
at $x$. 
Then the real link of $X$ along the stratum $V^\alpha$ is 
\begin{equation}
  \mathcal K(X,V^\alpha) = X \cap N_x \cap \partial B_\varepsilon(x)
  \label{eqn:RealLinkAlongAStratum}
\end{equation}
for $1 \gg \varepsilon >0$ sufficiently small. 
Similarly, the complex link is 
\begin{equation}
  \mathcal L(X,V^\alpha) = X \cap N_x \cap B_\varepsilon(x) \cap l^{-1}(\{\delta\})
  \label{eqn:ComplexLinkAlongAStratum}
\end{equation}
for a sufficiently general linear form $l$ and $1\gg \varepsilon \gg |\delta|>0$. 
For a rigorous definition of these objects see for example 
\cite[Part I, Chapter 1.4]{GoreskyMacPherson88} and 
\cite[Part II, Chapter 2.2]{GoreskyMacPherson88}. There, one can also find 
a proof of the fact that the real and complex links are independent 
of the choices involved in their definition: They are invariants 
of the particular stratification of $X$ and unique up to non-unique homeomorphism.

\medskip

In this note we shall also be concerned with the \textit{real and complex links 
of higher codimension} for a germ $(X,0)\subset (\CC^n,0)$ at the point $0$. 

\begin{definition}
  \label{def:RealAndComplexLinksOfHigherCodimension}
  The real and complex links of codimension $i$ of $(X,0) \subset (\CC^n,0)$ 
  at the origin are the classical real and complex links of a 
  section $(X \cap D_i, 0)$ of $X$ with a sufficiently general plane $D_i \subset \CC^n$ 
  of codimension $i$.
\end{definition}

\noindent
Note that in this definition, the classical complex link $\mathcal L(X,0)$ 
is the complex link of codimension $0$, even though it has complex codimension 
$1$ as an analytic subspace of $X$. This choice was made in order to be 
compatible with the notation for the real links. The reader may think of 
the word ``link'' as an indicator to increment the codimension by one to arrive 
at the actual codimension of the space.

By ``sufficiently general'' we mean that $D_i$ has to belong to 
a Zariski open subset $U_i \subset \Grass(n-i,n)$ in the Grassmannian. 
This set consists of planes for which variation of $D_i$ results in a 
Whitney equisingular family in $X \cap D_i$ and we shall briefly discuss 
the existence of such a set. From this it will follow immediately that the 
real and complex links of higher codimension are invariants 
of the germ $(X,0)$ itself and unique up to non-unique homeomorphism. 

Consider the \textit{Grassmann modification} of $X$:
\[
  G_i X = \{ (x,D) \in X \times \Grass(n-i,n) : x \in D \}
\]
where by $\Grass(n-i,n)$ we denote the Grassmannian of codimension $i$ planes 
in the ambient space $\CC^n$ through the point $0$. 
It comes naturally with two projections 
\begin{equation}
  \xymatrix{ 
    G_i X \ar[d]^{\rho} \ar[r]^-{\pi} &
    \Grass(n-i,n) \\
    X &
  }
  \label{eqn:GrassmannModification}
\end{equation}
where by construction $\rho(\pi^{-1}(\{D\})) = X \cap D$ for 
any plane $D \in \Grass(n-i,n)$. 
The Grassmannian itself has a natural embedding 
into $G_i X$ as the zero section $E_0$ of $\pi$ and we may think of $E_0$ 
as parametrizing the intersections of $X$ with planes of codimension $i$.
For $i\leq d = \dim X$ the intersection of any plane $D$ with $(X,0)$ will be of 
positive dimension at $0$ and 
therefore $E_0$ is necessarily contained in the closure of its complement.

Note that for an arbitrary point $p\in X$ the fiber $\rho^{-1}(\{p\})$ over $p$ in 
$G_i X$ consists of all planes $D \in \Grass(n-i,n)$ containing both $p$ and the origin. 
It follows that the restriction 
\[
  \rho \colon G_i X \setminus E_0 \to X \setminus \{0\}
\]
is not an isomorphism, but nevertheless a holomorphic fiber bundle with smooth fibers. 
In particular, any given Whitney stratification on $X$ 
determines a unique Whitney stratification of $G_i X \setminus E_0$ by pullback of 
the strata.

Given such a stratification on the complement of $E_0$, 
there exists some maximal Zariski open subset 
$U_i \subset E_0 \cong \Grass(n-i,n)$ such that all the strata of $G_i X \setminus E_0$ 
satisfy Whitney's conditions (a) and (b) along $U_i$. 
We may extend the stratification on the open subset $G_i X \setminus E_0$ to a Whitney 
stratification of $G_i X$, containing $U_i$ as a stratum. 
Since the Grassmannian $\Grass(n-i,n)$ is irreducible and complex 
analytic subsets have real codimension $\geq 2$, this stratum is necessarily dense 
and connected. 

A plane $D$ of codimension $i$ is sufficiently general 
in the sense of Definition \ref{def:RealAndComplexLinksOfHigherCodimension} if 
it belongs to $U_i$. 
The projection $\pi$ provides us with a canonical choice of normal slices 
along $U_i$ and by construction, we may therefore identify the real and complex 
links of $G_i X$ along $U_i$ with the real and complex links of the germ 
$(X \cap D,0) \subset (\CC^n,0)$ for any $D \in U_i$. Given that 
the classical real and complex links are invariants of the strata in a 
Whitney stratification, it follows immediately that the real and complex 
links of higher codimension for 
a germ $(X,0)$ are well defined invariants of the germ itself and 
unique up to non-unique homeomorphism.

\subsection{Synopsis with polar varieties}

While it was already established in the previous discussion that the real and 
complex links of higher codimension are invariants of the germ 
$(X,0) \subset (\CC^n,0)$ itself, a more specific setup will be 
required in the following. 
Let $\underline l = l_1,l_2,\dots,l_d \in \Hom(\CC^n,\CC)$ be a sequence of 
linear forms defining a flag of subspaces 
\[
   \mathcal D : \CC^n = D_0 \supset D_1 \supset D_2 \supset \dots \supset D_d \supset\{0\}
\]
in $\CC^n$ via $D_{i+1} = D_i \cap \ker l_{i+1}$. 

\begin{definition}
  \label{def:AdmissibleSequenceOfLinearForms}
  Let $(X,0) \subset (\CC^n,0)$ be an equidimensional reduced complex analytic 
  germ of dimension $d$ and $\{V^\alpha\}_{\alpha \in A}$ 
  a Whitney stratification of $X$. 
  A sequence of linear forms $l_1,l_2,\dots,l_d \in \Hom(\CC^n,\CC)$ 
  is called \textit{admissible} for $(X,0)$ if 
  for every $0 < i \leq d$ the restriction of 
  $l_i$ to $X_{i-1} := X \cap D_{i-1}$ does not annihilate any limiting 
  tangent space of $(X_{i-1},0)$ at the origin. 
\end{definition} 

This particular definition is chosen with a view towards the fibration 
theorems and inductive arguments used in Section \ref{sec:VariationSequence}.
However, we shall also need the results on polar varieties by L\^e and Teissier 
from \cite{LeTeissier81}. 
To this end, we recall the necessary definitions and explain how 
Definition \ref{def:AdmissibleSequenceOfLinearForms} fits into the context 
of their paper.

\medskip

The geometric setup for the treatment of limiting tangent spaces is 
the \textit{Nash modification}. 
Let $X \subset U$ be a suitable representative of $(X,0)$ in some open 
neighborhood $U$ of the origin. The Gauss map is defined on the regular locus $X_\reg$ via 
\[
  X_\reg \to \Grass(d,n), \quad p \mapsto [T_p X_\reg \subset T_p \CC^n].
\]
Then the Nash blowup $\tilde X$ of $X$ is defined as the closure of the graph of 
the Gauss map. It comes with two morphisms  
\[
  \xymatrix{
    \tilde X \ar[d]^\nu \ar[r]^-\gamma & 
    \Grass(d,n) \\
    X
  }
\]
where $\nu$ is the projection to $X$ and $\gamma$ the natural prolongation of the Gauss map. 
We denote by $\tilde T$ the pullback of the tautological bundle on $\Grass(d,n)$ along $\gamma$ 
and by $\tilde \Omega^1$ the dual of $\tilde T$. By construction, the fiber $\nu^{-1}(\{p\})$ over 
a point $p \in X$ consists of pairs $(p,E) \in \tilde X \subset \CC^n \times \Grass(d,n)$
where $E$ is a \textit{limiting tangent space} to $X$ at $p$. In particular, such a limiting 
tangent space is unique at regular points $p \in X_\reg$ and $\nu \colon \nu^{-1}(X_\reg) \to X_\reg$ 
is an isomorphism identifying the restriction of $\tilde T$ with the tangent bundle of $X_\reg$.

A flag $\mathcal D$ as above determines a set of \textit{degeneraci loci} in 
the Grassmannian as follows. Every linear form $l \in \Hom(\CC^n,\CC)$ can be 
pulled back to a global section $\nu^* l$ in the dual of the tautological bundle
of $\Grass(d,n)$: On a fiber $E$ of the tautological bundle the section $\nu^*l$ is 
defined to be merely the restriction of $l$ to $E$ regarded as a linear 
subspace of $\CC^n$.
This generalizes in the obvious way for the linear maps 
\[
  \varphi_k := l_1 \oplus l_2 \oplus \dots \oplus l_{d-k+1} \colon \CC^n \to \CC^{d-k+1}
\]
for every $0 \leq k \leq d$. Adapting the notation from \cite{LeTeissier81} we now have 
\begin{eqnarray*}
  c_k(\mathcal D) &:=&  \{ E \in \Grass(d,n) : \dim E \cap D_{d-k+1} \geq k \}\\
  &=& \{ E \in \Grass(d,n) : \rank \nu^* \varphi_k < d-k+1 \}.
\end{eqnarray*}
Note that by construction $c_0(\mathcal D) = \Grass(d,n)$ is the whole 
Grassmannian.
We set $\gamma^{-1}(c_k(\mathcal D)) \subset \tilde X$ to be the 
corresponding degeneraci locus on the Nash modification and 
\[
  P_k(\mathcal D) := \nu( \gamma^{-1}(c_k(\mathcal D)))
\]
its image in $X$. By construction, 
a point $p \in X$ belongs to $P_k(\mathcal D)$ if and only if 
there exists a limiting tangent space $E$ to $X_\reg$ at $p$ 
such that the restriction of $\varphi_k$ to $E$ does not have 
full rank.

The $k$-th polar multiplicity of $(X,0)$ is then defined 
for $0\leq k < d$ to be 
\begin{equation}
  m_k(X,0) := m_0( P_{k}(\mathcal D)),
  \label{eqn:DefinitionPolarMultiplicity}
\end{equation}
the multiplicity of the $k$-th polar variety.
We shall see below that for $k=d$ the variety $P_d(\mathcal D)$ 
is empty for a generic flag so that a definition of 
$m_d(X,0)$ does not make sense. In the other extreme case where $k=0$, 
the polar multiplicity $m_0(X,0)$ is simply the multiplicity of $(X,0)$ 
itself.

\begin{lemma}
  \label{lem:CompatibilityLemmaForPolarVarieties}
  Let $X$ be a sufficiently small representative of $(X,0) \subset (\CC^n,0)$. 
  A sequence of linear forms $l_1,\dots,l_d$ is admissible for $(X,0)$ 
  if and only if for every $0\leq i < d$ one has
  \begin{equation}
    \gamma^{-1}(c_{d-i}(\mathcal D)) \cap \widehat X_i = \emptyset
    \label{eqn:CompatibilityEquation}
  \end{equation}
  for the associated flag $\mathcal D$.  
  Here $X_i := X \cap D_i$ and $\widehat X_i$ denotes the 
  strict transform of $X_i$ in the Nash modification. 
\end{lemma}

\begin{proof}
  The proof will proceed by induction on $i$. For $i=0$ we have 
  $\widehat X_0 = \tilde X$ and the statement is about 
  the choice of $l_1$. 
  Consider the projectivized analytic 
  set of degenerate covectors 
  \[
    \{ ([l],(p,E)) \in \PP\Hom(\CC^n,\CC)\times \nu^{-1}(\{0\}) : l|_E = 0 \}
  \]
  along the central fiber $\nu^{-1}(\{0\})$ of the Nash modification. 
  Since this fiber has strictly smaller dimension than $X$, 
  the set of degenerate covectors has dimension $<n-1$ and 
  therefore the discriminant, i.e. the image of its projection to $\PP\Hom(\CC^n,\CC)$, 
  is a closed analytic set of positive codimension. Now (\ref{eqn:CompatibilityEquation}) 
  is satisfied if and only if $l_1$ belongs to the complement of the affine cone 
  of the discriminant. 

  Such a choice for $l_1$ determines $X_1 = X \cap D_1$.
  Note that by construction this intersection is transversal and therefore 
  $X_1$ inherits a Whitney stratification from the one 
  on $X$. Moreover, at a regular point $p \in X_1$ the tangent space 
  \[
    T_p X_1 = \ker \nu^*l_1 \subset T_p X
  \]
  is naturally contained in the tangent space of $X$ to that point. 
  Let $\widehat{X}_1 \subset \tilde X$ be the strict transform of 
  $X_1$ in the Nash modification. Taking limits of appropriate (sub-)sequences 
  of regular points, it is easy to see that every limiting tangent space 
  $E'$ of $X_1$ at $0$ is contained in a limiting tangent space $E$ of 
  $X$ along $X_1$. Consequently, a second linear form $l_2$ annihilates 
  the limiting tangent space $E'$ of $X_1$ if and only if $l_1 \oplus l_2$ 
  is degenerate on $E$. But this means nothing else than 
  \[
    (0,E) \in \widehat{X}_1 \cap \gamma^{-1}(c_{d-1}(\mathcal D)) \neq \emptyset 
  \]
  which establishes the claim for $i=1$. 
  The remainder of the induction is a repetition of the previous steps and left 
  to the reader.
\end{proof}

The previous lemma provides the link of Definition \ref{def:AdmissibleSequenceOfLinearForms} 
with the ``Th\'eor\`eme de Bertini id\'ealiste'' by L\^e and Teissier 
\cite[Th\'eor\`eme 4.1.3]{LeTeissier81}. They establish the existence of  
Zariski open subsets $U'_i \subset \Grass(n-i,n)$ with certain good 
properties concerning the variety $\gamma^{-1}(c_{d-i+1}(D_i))$ for 
$D_i \in U'_i$. 
A posteriori, they discuss in \cite[Proposition 4.1.5]{LeTeissier81} 
that if the whole flag $\mathcal D$ has been chosen such that $D_i \subset U'_i$ 
for all $i$, then also (\ref{eqn:CompatibilityEquation}) is in fact satisfied for all $i$. 
Thus we obtain the following

\begin{corollary}
  \label{cor:ParticularChoiceOfAdmissibleLinearForms}
  For every equidimensional reduced analytic germ $(X,0)$ there exists 
  a Zariski open and dense subset of admissible 
  sequences of linear forms $l_1,l_2\dots,l_d$. Moreover, this sequence can be chosen 
  such that for the associated flag $\mathcal D$ the space
  $D_i = \{ l_1=\dots=l_i=0\}$ is sufficiently general in the sense 
  of Definition \ref{def:RealAndComplexLinksOfHigherCodimension} so that 
  the real and complex links of codimension $i$ are given by 
  \[
    \mathcal K^i(X,0) = X \cap D_i \cap \partial B_\varepsilon(0) 
  \]
  and 
  \[
    \mathcal L^i(X,0) = X \cap D_i \cap B_\varepsilon(0) \cap l_{i+1}^{-1}(\{\delta\})
  \]
  for $1 \gg \varepsilon \gg |\delta| >0$, respectively.
\end{corollary}

\begin{proof}
  Consider the sets $U'_i \cap U_i$ with $U'_i$ from L\^e's and Teissier's 
  ``Th\'eor\`eme de Bertini id\'ealiste'' and $U_i$ the set of Whitney 
  equisingular sections from the discussion of Definition 
  \ref{def:RealAndComplexLinksOfHigherCodimension}. Since the intersection of 
  Zariski open sets is again Zariski open, we may choose $l_1,l_2,\dots,l_d$ 
  to be any sequence of linear forms such that 
  \[
    D_i = \{ l_1 = \dots = l_i = 0\} \subset U'_i \cap U_i
  \]
  for all $i$. 
\end{proof}

We will henceforth assume that the sequence of linear forms $l_i$ 
has been chosen such that the associated flag $\mathcal D$ has 
$D_i \in U'_i \cap U_i$ where $U'_i$ is the 
Zariski open subset of L\^e's and Teissiers' ``Th\'eor\`eme de Bertini id\'ealiste'' 
and $U_i$ the Zariski open subset of Whitney equisingular sections 
of codimension $i$ from the discussion of Definition \ref{def:RealAndComplexLinksOfHigherCodimension}.

\subsection{The Euler characteristic of complex links} 
\label{sec:EulerCharacteristicComplexLinks}

L\^e and Teissier have described a method to compute the 
Euler characteristic of complex links from the polar multiplicities
in \cite[Proposition 6.1.8]{LeTeissier81}. We briefly sketch 
how to use their results inductively in our setup from 
Definition \ref{def:AdmissibleSequenceOfLinearForms}.

\medskip

As before, let $(X,0) \subset (\CC^n,0)$ be a reduced, equidimensional 
complex analytic germ of dimension $d$, endowed with a Whitney stratification 
$\{ V^\alpha\}_{\alpha \in A}$. We will assume that $V^0 = \{0\}$ is a stratum 
and write $X^\alpha = \overline{V^\alpha}$ for the closure of any other stratum 
$V^\alpha$ of $X$. Throughout this section, we will assume that an admissible 
sequence of linear forms $l_1,\dots,l_d$ and the associated flag $\mathcal D$ 
have been chosen as in Corollary \ref{cor:ParticularChoiceOfAdmissibleLinearForms}
for all germs $(X^\alpha,0)$ at once. This flag being fixed, we will in the 
following suppress it from our notation and simply write 
$P_k(X^\alpha,0)$ for the polar varieties of the germ $(X^\alpha,0)$ with 
$\mathcal D$ being understood. 

Denote by $\mathcal L(X,V^\alpha)$ the classical complex links of $X$ along 
the stratum $V^\alpha$ and by $\mathcal L^i$ the complex link of codimension 
$i$ of $X$ at the origin. Then by \cite[Th\'eor\`eme 6.1.9]{LeTeissier81} 
\begin{equation}
  \chi\left( \mathcal L^0 \right) - \chi\left( \mathcal L^1 \right) = 
  \sum_{\alpha \neq 0} m_0\left( 
  P_{d(\alpha)-1} \left( X^\alpha,0 \right)\right)\cdot 
  (-1)^{d(\alpha)-1} 
  \left( 1 - \chi(\mathcal L(X,V^\alpha)) \right).
  \label{eqn:InductionFormulaForEulerCharacteristics}
\end{equation}
For our specific setup we may interpret this formula in 
the context of stratified Morse theory. To this end, 
note that 
the restriction of $l_2$ to the complex link $\mathcal L^0$ 
\[
  l_2 \colon X \cap B_\varepsilon(0) \cap l_1^{-1}(\{\delta_1\}) \to \CC
\]
is a stratified Morse function with critical points 
on the interior of the strata $V^\alpha \cap \mathcal L^0$ of $\mathcal L^0$ 
precisely at the intersection points 
\[
  \mathcal L^0 \cap P_{d(\alpha)-1}(X^\alpha,0) = \{ q^\alpha_1,\dots,q^\alpha_{m_0(P(\alpha)_{d-1}(X^\alpha,0))}\}, 
\]
cf. \cite[Corollaire 4.1.6]{LeTeissier81} and \cite[Corollaire 4.1.9]{LeTeissier81}.
The complex link of codimension $2$ can be identified with the general fiber 
\[
  \mathcal L^1 \cong l_{2}^{-1}(\{\delta_2\}) \cap \mathcal L^0
\]
for some regular value $\delta_2$ off the discriminant and we can use 
the function $\lambda = |l_2-\delta_2|^2$ as a Morse function in order 
to reconstruct $\mathcal L^0$ from $\mathcal L^1$. It can be 
shown that for suitable choices of the represenatives involved, the 
critical points of $\lambda$ on the boundary of $\mathcal L^0$ are 
``outward pointing'' and hence do not contribute to changes in topology; 
see for instance \cite{Zach17} or \cite{PenafortZach18} for a discussion. 
For an interior critical point $q^\alpha_j \in V^\alpha \cap \mathcal L^0$ we have 
the product of the \textit{tangential} and the \textit{normal Morse data}
\[
  \left( D^{d(\alpha)-1}, \partial D^{d(\alpha)-1} \right) \times 
  \left( C(\mathcal L(X,V^\alpha)), \mathcal L(X,V^\alpha) \right)
\]
where $D^{d(\alpha)}$ is the disc of real dimension $d(\alpha) = \dim_\CC V^\alpha$ 
and $C(\mathcal L(X,V^\alpha))$ the cone over the complex link of $X$ 
along $V^\alpha$. It is a straighforward calculation that the Euler characteristic 
changes precisely by $(-1)^{d(\alpha)-1} (1 - \chi(\mathcal L(X,V^\alpha)))$ 
for the attachement of this cell at any of the critical points $q_j^\alpha$. 
Summation over all these points on all relevant strata therefore gives us 
back the Formula (\ref{eqn:InductionFormulaForEulerCharacteristics}) 
by L\^e and Teissier.

\medskip

It is evident that the above procedure can be applied inductively, 
cf. \cite[Remarque 6.1.10]{LeTeissier81}. This allows to reconstruct 
the codimension $i$ complex link $\mathcal L^i$ of $(X,0)$ from its 
hyperplane sections 
\[
  \mathcal L^{i} \supset \mathcal L^{i+1} \supset \dots \supset 
  \mathcal L^{d-1} = \{x_1,\dots,x_{m_0(X,0)}\}
\]
starting with $\mathcal L^{d-1}$ which is just a set of points whose 
number is equal to the multiplicity $m_0(X,0)$ of $(X,0)$ at the origin. 
We leave it to the reader to verify the formula 
\begin{equation}
  \chi( \mathcal L^i ) = 
  \sum_{\alpha \in A}
  \left( \sum_{j=i+1}^{d(\alpha)}(-1)^{d(\alpha)-j} m_0\left( P_{d(\alpha)-j}(X^\alpha,0) \right) \right)
  \cdot 
  \left( 1 - \chi(\mathcal L(X,V^\alpha) \right).
  \label{eqn:FormulaForEulerCharacteristic}
\end{equation}
The coefficients appearing in this formula are nothing but 
the \textit{local Euler obstruction} of $(X^\alpha \cap D_{i-1},0)$ at the origin:
\begin{equation}
  \mathrm{Eu}(X^\alpha \cap D_{i-1},0) = \sum_{j=i}^{d(\alpha)}(-1)^{d(\alpha)-j} 
  m_0\left( P_{d(\alpha)-j}(X^\alpha,0)\right)
  \label{eqn:EulerObstructionAndPolarMultiplicities}
\end{equation}
cf. \cite[Corollaire 5.1.2]{LeTeissier81}.

\subsection{Polar multiplicities of generic determinantal varieties}
\label{sec:PolarMultiplicities}
We now turn towards the study of the generic 
determinantal varieties $M_{m,n}^s \subset \CC^{m\times n}$. 
These are equipped with the \textit{rank stratification}, i.e. the decomposition 
\[
  M_{m,n}^s = \bigcup_{r<s} V_{m,n}^r, \quad 
  V_{m,n}^r = \{ \varphi \in \CC^{m\times n} : \rank \varphi = r \}.
\]
Due to its local analytic triviality, this stratification is easily seen 
to satisfy both Whitney's conditions (a) and (b).

The reduced Euler characteristics of the classical complex links have been computed by Ebeling 
and Gusein-Zade in 
\cite[Proposition 3]{EbelingGuseinZade09}:
\begin{equation}
  1 - \chi( \mathcal L(M_{m,n}^s,0)) = (-1)^{s-1} { m-1 \choose s-1 },
  \label{eqn:EbelingGuseinZadeFormulaEulerCharacteristic}
\end{equation}
where, without loss of generality, it is assumed that $m \leq n$. 
The generic determinantal varieties admit a recursive pattern in the 
following sense. For $r<s\leq m \leq n$, a normal slice to 
the stratum $V_{m,n}^r \subset M_{m,n}^s$ 
through the point $\one_{m,n}^r$ is given by the set of matrices of the form 
\[
  N_{m,n}^r = \one^r \oplus \CC^{(m-r) \times (n-r)} = 
  \begin{pmatrix}
    \one^r & 0 \\
    0 & \CC^{(m-r)\times(n-r)}
  \end{pmatrix} \subset \CC^{m\times n}.
\]
It follows immediately that $\mathcal L(M_{m,n}^s,V_{m,n}^r) \cong \mathcal L(M_{m-r,n-r}^{s-r},0)$ 
and hence 
\begin{equation}
  1-\chi( \mathcal L(M_{m,n}^s, V_{m,n}^r)) 
  = (-1)^{s-r-1}{ m-r-1 \choose s-r-1 }.
  \label{eqn:EbelingGuseinZadeFormulaEulerCharacteristicGeneralized}
\end{equation}
In order to determine the topological Euler characteristic of the 
complex links of higher codimension $\mathcal L^i(M_{m,n}^s,0)$
of the generic determinantal varieties by means of the previous section, 
we need to know all the relevant polar multiplicities 
\begin{equation}
  e_{m,n}^{r,k} := m_k(M_{m,n}^{r+1},0) = m_0 \left( P_k(M_{m,n}^{r+1},0) \right).
  \label{eqn:DefinitionDeterminantalMultiplicities}
\end{equation}
There are several methods to achieve this. For instance, one could simply 
choose random linear forms $l_i$ and compute the resulting multiplicity 
with the aid of a computer algebra system using e.g. Serre's intersection formula. 
However, this approach provides very little insight and the results a priori 
depend on the choice of the linear forms. 
Recently, X. Zhang has computed 
the polar multiplicities in 
\cite{Zhang17} 
using Chern class calculus which is an exact computation not 
depending on any particular choices.
His formulas, however, are very complicated 
since they appear as byproducts of the study of the Chern-Mather classes 
of determinantal varieties. 
In this section we will follow a more direct approach to the computation 
of the polar multiplicities using Chern classes. 

\medskip
In \cite[Th\'eor\`eme 5.1.1]{LeTeissier81}, L\^e and Teissier give the following 
formula for the polar multiplicities of a germ $(X,0) \subset (\CC^n,0)$: 
\begin{equation}
  m_0\left( P_{k}(\mathcal D) \right)
  (-1)^{d-1} 
  \int_{\mathfrak Y} c_{k}(\tilde T) \cdot c_1(\OO(1))^{d-k-1}.
  \label{eqn:PolarMultiplicityAsIntegral}
\end{equation}
Here the integral is taken over the exceptional divisor $\mathfrak Y$ 
of the blowup $\mathfrak X$ of the Nash modification $\tilde X$ along the pullback 
of the maximal ideal at the origin for $(X,0) \subset (\CC^n,0)$ and $\mathcal O(1)$ 
denotes the dual of the tautological bundle for that blowup. 
By construction, these spaces can be arranged in a commutative diagram 
\begin{equation}
  \xymatrix{ 
    \mathfrak Y \ar[r] \ar@{->}[dr] \ar[d] & 
    \nu^{-1}(\{0\}) \ar@{->}[dr] & 
    \\
    \pi^{-1}(\{0\}) \ar@{->}[dr] & 
    \mathfrak X \ar[r] \ar[d] & 
    \tilde X \ar[d]^\nu \\
    & 
    \mathrm{Bl}_0 X \ar[r]^\pi & 
    X.
  }
  \label{eqn:CommutativeDiagramWithAllBlowups}
\end{equation}
where $\mathrm{Bl}_0 X$ denotes the usual blowup of $X$ at the origin.
We will now describe this diagram for the particular case where 
$(X,0) = (M_{m,n}^s,0) \subset (\CC^{m\times n},0)$ is a generic determinantal 
variety and deduce our particular formula (\ref{eqn:FormulaForThePolarMultiplicity}) from that.
\medskip

The Nash blowup of $(M_{m,n}^s,0)\subset(\CC^{m\times n},0)$ has been studied by Ebeling and Gusein-Zade 
in \cite{EbelingGuseinZade09} and we briefly review their discussion.
Let $r=s-1$ be the rank of the matrices in the open stratum $V_{m,n}^{r} \subset M_{m,n}^s$. 
The tangent space to $V_{m,n}^r$ at a point $\varphi$ is known to be 
\begin{equation}
  T_\varphi V_{m,n}^{r} = \{ \psi \in \Hom( \CC^n, \CC^m ) : \psi( \ker \varphi ) 
    \subset \mathrm{im}\, \varphi \}.
  \label{eqn:TangentBundleSmoothStratum}
\end{equation}
This fact can be exploited to replace the Grassmannian
for the Nash blowup of $M_{m,n}^{s}$ by a product: 
Let $\Grass(r,m)$ be the Grassmannian of $r$-planes in $\CC^m$ and 
$\Grass(r,n)$ the Grassmannian of $r$-planes in $(\CC^n)^\vee$, the dual of $\CC^n$. 
Then the Gauss map factors through 
\[
  \hat \gamma : V_{m,n}^r \to \Grass(r,n) \times \Grass(r,m), \quad
  \varphi \mapsto \left( (\ker \varphi)^\perp, \mathrm{im}\, \varphi \right),
\]
where by $(\ker \varphi)^\perp$ we mean the linear forms in 
$(\CC^n)^\vee$ vanishing on $\ker M \subset \CC^n$. 
We denote this double Grassmannian by $G = \Grass(r,n) \times \Grass(r,m)$. 
On $G$ we have the two tautological exact sequences 
$0 \to S_1 \overset{i_1}{\longrightarrow} \OO^n \overset{\pi_1}{\longrightarrow} Q_1 \to 0$ and 
$0 \to S_2 \overset{i_2}{\longrightarrow} \OO^m \overset{\pi_2}{\longrightarrow} Q_2 \to 0$
coming from the two factors 
with $S_2$ corresponding to the images and $Q_1^\vee$ to the kernels 
of the matrices $\varphi$ under the modified Gauss map $\hat \gamma$.
With this notation, the space of matrices 
\[
  \Hom(\CC^n,\CC^m) \cong (\CC^n)^\vee \otimes \CC^m
\]
pulls back to the trivial bundle $\OO^n \otimes \OO^m$ from which 
we may project to the product 
\[
  \pi = \pi_1 \otimes \pi_2 \colon \OO^n \otimes \OO^m \to Q_1 \otimes Q_2.
\]
Then the condition on $\psi$ in (\ref{eqn:TangentBundleSmoothStratum}) 
becomes $\pi( \psi ) = 0$ and consequently the Nash bundle on $\tilde M_{m,n}^s$ 
is given by 
\begin{equation}
  \tilde T = \hat \gamma^* \left( \ker \pi \right).
  \label{eqn:NashBundleOnDeterminantalVariety}
\end{equation}

The Nash transform $\tilde M_{m,n}^s$ itself can easily be seen to be isomorphic to 
the total space of the vector bundle 
\begin{equation}
  \tilde M_{m,n}^s \cong \left| \Hom\left( (S_1)^\vee, S_2 \right) \right| 
  = \left| S_1 \otimes S_2 \right|
  \label{eqn:NashTransformAsVectorBundle}
\end{equation}
on $G$. In particular, $\tilde M_{m,n}^s$ is smooth and the maximal ideal 
of the origin in $\CC^{m\times n}$ pulls back to the ideal sheaf of the 
zero section in $\tilde M_{m,n}^s$. Thus the exceptional divisor 
$\mathfrak Y$ in (\ref{eqn:CommutativeDiagramWithAllBlowups}), 
i.e. the domain of integration in (\ref{eqn:PolarMultiplicityAsIntegral}), 
is nothing but the projectivized bundle 
\[
  \mathfrak Y = \PP \tilde M_{m,n}^s = \PP \left( S_1\otimes S_2 \right).
\]

\begin{proposition}
  \label{prp:FormulaForThePolarMultiplicity}
  The $k$-th polar multiplicity of $(M_{m,n}^{r+1},0) \subset (\CC^{m\times n},0)$ 
  is given by 
  \begin{equation}
    e_{m,n}^{r,k} = (-1)^{(m+n)\cdot r - r^2 - 1} 
    \int_{G} s_{k}\left( Q_1 \otimes Q_2 \right) s_{(m+n)r - 2r^2 -k}(S_1 \otimes S_2)
    \label{eqn:FormulaForThePolarMultiplicity}
  \end{equation}
  where $G = \Grass(r,n) \times \Grass(r,m)$, $S_i$ and $Q_i$ are the 
  tautological sub- and quotient bundles coming from either one of the two factors, 
  and $s_k$ denotes the $k$-th Segre class.
\end{proposition}

\begin{proof}
  Starting from the L\^e-Teissier formula (\ref{eqn:PolarMultiplicityAsIntegral}) 
  we substitute the different terms according to the above identifications for 
  the generic determinantal varieties.
  From (\ref{eqn:NashBundleOnDeterminantalVariety}) we see that for every $k \geq 0$ 
  \[
    c_k(\tilde T) = s_k(Q_1 \otimes Q_2) 
  \]
  is the $k$-th Segre class of the complement 
  $Q_1 \otimes Q_2$ of $\tilde T$ in $\OO^n \otimes \OO^m$.
  The integral in question becomes 
  \[
    \int_{\PP(S_1 \otimes S_2)} s_{k}(Q_1\otimes Q_2) \cdot c_1(\OO(1))^{(m+n)r - r^2-1-k}
  \]
  and we may perform this integration in two steps with the 
  first one being integration along the fibers of the projection 
  $\PP(S_1 \otimes S_2) \to G$. Since $\OO(1)$ denotes the dual 
  of the tautological bundle for the projectivization of the underlying 
  vector bundle $S_1 \otimes S_2$, the result now follows from the projection 
  formula, cf. \cite[Chapter 3.1]{Fulton98}.
\end{proof}

Formula (\ref{eqn:FormulaForThePolarMultiplicity}) can be implemented
in computer algebra systems such as Singular \cite{Sin}. For instance,
using the library \verb|schubert.lib|, the computations can easily be
carried out for $m,n\leq 5$ on a desktop computer.  Series such as,
for example, $e_{7,8}^{4,k}$ are also still feasible, but take up to 9
minutes to finish. 

We have computed several polar multiplicities of the generic determinantal 
varieties using this formula. The results are listed in 
Tables \ref{tab:PolarMultiplicities2byNRank1} -- \ref{tab:PolarMultiplicities5byNRank1and4}.

\begin{remark}
  The reader will also note a symmetry that first appears in 
  Table \ref{tab:PolarMultiplicities3byNRank1}: For $0\leq r \leq m\leq n$ 
  the polar multiplicities satisfy 
  \begin{equation}
    e_{m,n}^{r,k} = e_{m,n}^{m-r,2(m-r)r-k}.
    \label{eqn:DualityPolarMultiplicities}
  \end{equation}
  This phenomenon is based on a duality of the (projectivized) conormal modifications 
  of the generic determinantal varieties, as has been 
  explained to the author by Terence Gaffney in an oral communication. 
  Formula (\ref{eqn:DualityPolarMultiplicities}) then follows 
  from \cite[Theorem 3.3]{Urabe81}.
  Interestingly, we have not succeeded to derive this symmetry from 
  the formula in (\ref{eqn:FormulaForThePolarMultiplicity}), but we 
  nevertheless use it in the following tables in order to not duplicate 
  the statements.
\end{remark}

\begin{remark}
  The computation of Chern- and Segre classes of tensor products of 
  vector bundles is a surprisingly expensive task from a computational 
  point of view. Different formulas and algorithms have, for instance, been implemented 
  in the Singular libraries \verb|chern.lib|; see also \cite{Zsolt19} for a 
  further discussion. 

  The $k$-th Segre class of the tensor product $S_1 \otimes S_2$ of the two tautological 
  subbundles on $G = \Grass(r,n)\times \Grass(r,m)$ is the restriction of a universal 
  polynomial 
  \[
    P_k
    \in \ZZ[c_1(S_1),\dots,c_r(S_1), c_1(S_2),\dots,c_r(S_2)]
  \]
  in the Chern classes of the tautological bundles on the product of 
  infinite Grassmannians $\bigcup_{m,n\in \NN} \Grass(r,n)\times \Grass(r,m)$.
  However, these polynomials are not sparse and their degree in the Chern roots 
  is bounded only by $r^2$. Given that we have $2r$ variables, the number of 
  coefficients of the polynomials $P_k$ 
  can roughly be estimated by ${ r^2 + 2r -1 \choose 2r - 1}$. Already for values for 
  $r\geq 10$ we can therefore expect to have flooded the full RAM of any modern 
  desktop computer. 

  Compared to that, 
  the explicit model for the cohomology of $\Grass(r,n)$ introduced above leads 
  to an algebra of dimension ${m \choose r}{n\choose r}$ for the cohomology of 
  $G$. This number is in general strictly smaller than the number of 
  coefficients of $P_k$.
  Since we shall only need the results for fixed values of $m$ and $n$, it 
  seems likely that a manual implementation of a modular approach, 
  using the ideals $J_{m,r}$ introduced above, 
  could produce some further results for the polar multiplicities $e_{m,n}^{r,k}$ 
  which can not be reached using the methods provided by \verb|schubert.lib| and 
  \verb|chern.lib|.

  Other than that, it would, of course, be even more appealing to find a closed 
  formula for the polar multiplicities as a function of $m,n,r$, and $k$.
\end{remark}

\begin{table}
  \centering
  \begin{tabular}{|r|ccc|}
    \hline
    $k:$ & $0$ & $1$ & $2$ \\
    \hline
    $2\times 2$ & 2 & 2 & 2 \\
    $2\times 3$	& 3 & 4 & 3 \\
    $2\times 4$ & 4 & 6 & 4 \\
    $2\times 5$ & 5 & 8 & 5 \\
    $2\times 6$ & 6 & 10 & 6\\
    $2\times 7$ & 7 & 12 & 7\\
    \hline
  \end{tabular}
  \medskip
  \caption{The Polar multiplicities $e_{2,n}^{1,k}$ of $2\times n$-matrices for $n\leq7$; 
  values for $k$ which are not explicitly listed, are zero.}
  \label{tab:PolarMultiplicities2byNRank1}
\end{table}

\begin{table}
  \centering
  \begin{tabular}{|r|ccccc|ccccc|}
    \hline 
    & & & $e_{3,n}^{1,k}$ & & & & & $e_{3,n}^{2,k}$ & & \\
    \hline
    $k:$ & 
    $0$ & $1$ & $2$ & $3$ & $4$ &
    $0$ & $1$ & $2$ & $3$ & $4$ \\
    \hline
      $3 \times 3$ & 
      6 & 12 & 12 & 6 & 3 &
      3 & 6 & 12 & 12 & 3\\
      $3 \times 4$ & 
      10 & 24 & 27 & 16 & 6 &
      6 & 16 & 27 & 24 & 10 \\
      $3 \times 5$ & 
      15 & 40 & 48 & 30 & 10 &
      10 & 30 & 48 & 40 & 15 \\
      $3 \times 6$ & 
      21 & 60 & 75 & 48 & 15 &
      15 & 48 & 75 & 60 & 21 \\
      $3 \times 7$ & 
      28 & 84 & 108 & 70 & 21 & 
      21 & 70 & 108 & 84 & 28 \\
      $3 \times 8$ & 
      36 & 112 & 147 & 96 & 28 & 
      28 & 96 & 147 & 112 & 36 \\
      $3 \times 9$ & 
45& 144& 192& 126& 36 &
36& 126& 192& 144& 45\\ 
      $3 \times 10$ & 
55& 180& 243& 160& 45& 
45& 160& 243& 180& 55\\ 
      $3 \times 11$ & 
66& 220& 300& 198& 55& 
55& 198& 300& 220& 66\\ 
      $3 \times 12$ & 
78& 264& 363& 240& 66& 
66& 240& 363& 264& 78\\ 
      $3 \times 13$ & 
91& 312& 432& 286& 78&
78& 286& 432& 312& 91\\
      $3 \times 14$ & 
105& 364& 507& 336& 91&
91& 336& 507& 364& 105\\ 
      $3 \times 15$ & 
120& 420& 588& 390& 105&
105& 390& 588& 420& 120 \\
      $3 \times 16$ & 
136& 480& 675& 448& 120&
120 & 448 & 675 & 480 & 136 \\
      $3 \times 17$ & 
153& 544& 768& 510& 136& 
136 & 510 & 768 & 554 & 153 \\
      $3 \times 18$ & 
171& 612& 867& 576& 153&
153 & 576 & 876 & 612 & 171\\
      $3 \times 19$ & 
190& 684& 972& 646& 171&
171 & 646 & 972 & 684 & 190 \\
      $3 \times 20$ & 
210& 760& 1083& 720& 190 &
190 & 720 & 1038 & 760 & 210 \\
      \hline
    \end{tabular}
    \medskip
  \caption{Polar multiplicities for $3\times n$-matrices for $n \leq 20$;
  all values for $k$ which are not explicitly listed are zero.}
  \label{tab:PolarMultiplicities3byNRank1}
\end{table}
\begin{landscape}
\begin{table}
  \centering
  \begin{tabular}{|r|ccccccccc|ccccccc|}
    \hline 
    & 
    \multicolumn{8}{c}{$e_{4,n}^{2,k}$ with $k$ running from the left to the right} & & 
    \multicolumn{6}{c}{$e_{4,n}^{1,k}$ with $k$ running from the left to the right} & \\
    \hline
    $k:$ & 
    0 & 1 & 2 & 3 & 4 & 5 & 6 & 7 & 8 & 
    0 & 1 & 2 & 3 & 4 & 5 & 6 \\
    \hline
$4 \times 4$ & 
20& 80& 176& 256& 286& 256& 176& 80& 20& 
20& 60& 84& 68& 36& 12& 4 \\
$4 \times 5$ & 
50& 240& 595& 960& 1116& 960& 595& 240& 50& 
35& 120& 190& 176& 105& 40& 10 \\
$4 \times 6$ & 
105& 560& 1488& 2520& 2980& 2520& 1488& 560& 105& 
56& 210& 360& 360& 228& 90& 20 \\
$4 \times 7$ & 
196& 1120& 3115& 5432& 6488& 5432& 3115& 1120& 196& 
84& 336& 609& 640& 420& 168& 35\\
$4 \times 8$ & 
336& 2016& 5792& 10304& 12390& 10304& 5792& 2016& 336& 
120& 504& 952& 1036& 696& 280& 56\\
$4 \times 9$ & 
540& 3360& 9891& 17856& 21576& 17856& 9891& 3360& 540& 
165& 720& 1404& 1568& 1071& 432& 84\\
$4 \times 10$ & 
825& 5280& 15840& 28920& 35076& 28920& 15840& 5280& 825& 
220& 990& 1980& 2256& 1560& 630& 120\\
$4 \times 11$ & 
1210& 7920& 24123& 44440& 54060& 44440& 24123& 7920& 1210&
286& 1320& 2695& 3120& 2178& 880& 165 \\
$4 \times 12$ & 
1716& 11440& 35280& 65472& 79838& 65472& 35280& 11440& 1716& 
364& 1716& 3564& 4180& 2940& 1188& 220 \\
$4 \times 13$ & 
- & - & - & - & - & - & - & - & - & 
455& 2184& 4602& 5456& 3861& 1560& 286 \\
$4 \times 14$ & 
- & - & - & - & - & - & - & - & - & 
560& 2730& 5824& 6968& 4956& 2002& 364 \\
$4 \times 15$ & 
- & - & - & - & - & - & - & - & - & 
680& 3360& 7245& 8736& 6240& 2520& 455\\
$4 \times 16$ & 
- & - & - & - & - & - & - & - & - & 
816& 4080& 8880& 10780& 7728& 3120& 560 \\
$4 \times 17$ & 
- & - & - & - & - & - & - & - & - & 
969& 4896& 10744& 13120& 9435& 3808& 680\\
$4 \times 18$ & 
- & - & - & - & - & - & - & - & - & 
1140& 5814& 12852& 15776& 11376& 4590& 816 \\
$4 \times 19$ & 
- & - & - & - & - & - & - & - & - & 
1330& 6840& 15219& 18768& 13566& 5472& 969 \\
$4 \times 20$ & 
- & - & - & - & - & - & - & - & - & 
1540& 7980& 17860& 22116& 16020& 6460& 1140 \\
$4 \times 21$ & 
- & - & - & - & - & - & - & - & - & 
1771& 9240& 20790& 25840& 18753& 7560& 1330 \\
$4 \times 22$ & 
- & - & - & - & - & - & - & - & - & 
2024& 10626& 24024& 29960& 21780& 8778& 1540 \\
$4 \times 23$ & 
- & - & - & - & - & - & - & - & - & 
2300& 12144& 27577& 34496& 25116& 10120& 1771 \\
$4 \times 24$ & 
- & - & - & - & - & - & - & - & - & 
2600& 13800& 31464& 39468& 28776& 11592& 2024 \\
    \hline
    $l:$ & 
    & & & & & & & & &
    6 & 5 & 4 & 3 & 2 & 1 & 0 \\
    \hline
    & 
    & & & & & & & & &
    \multicolumn{6}{c}{$e_{4,n}^{3,l}$ with $l$ running from the right to the left} & \\
    \hline 
  \end{tabular}
  \medskip
  \caption{Polar multiplicities for $4\times n$-matrices. A - indicates that this value has 
  not been computed; all other entries for $k$ and $l$
which are not explicitly listed, are equal to zero.}
  \label{tab:PolarMultiplicities4byN}
\end{table}
  
\end{landscape}
\begin{landscape}
\begin{table}
  \centering
  \begin{tabular}{|r|ccccccccccccc|}
    \hline 
    & 
    \multicolumn{12}{c}{$e_{5,n}^{2,k}$ with $k$ running from the left to the right} & 
    \\
    \hline
    $k:$ & 
    0 & 1 & 2 & 3 & 4 & 5 & 6 & 7 & 8 & 9 & 10 & 11 & 12 
    \\
    \hline 
$5 \times 5$ & 
175& 1050& 3180& 6320& 9180& 10320& 9360& 7080& 4545& 2430& 1020& 300 & 50 \\
$5 \times 6$ & 
490& 3360& 11445& 25396& 40890& 50520& 49495& 39120& 24981& 12640& 4830& 1260& 175 \\ 
$5 \times 7$ & 
1176& 8820& 32480& 77280& 132300& 172074& 175080& 141120& 89880& 44310& 16128& 3920& 490 \\
$5 \times 8$ & 
2520& 20160& 78498& 196080& 349860& 470400& 489930& 399504& 253980& 123200& 43470& 10080& 1176\\
$5 \times 9$ & 
4950& 41580& 168840& 437220& 803916& 1106640& 1171360& 962640& 611100& 293076& 101160& 22680& 2520 \\
$5 \times 10$ & 
9075& 79200& 332310& 884840& 1664685& 2332440& 2498535& 2064960& 1309290& 622560& 211365& 46200& 4950 \\
$5 \times 11$ & 
15730& 141570& 609840& 1660296& 3180705& 4518690& 4885440& 4055040& 2568456& 1213080& 406560& 87120& 9075 \\
$5 \times 12$ & 
26026& 240240& 1057485& 2931760& 5699760& 8188224& 8918470& 7427640& 4700460& 2207920& 732303& 154440& 15730 \\
    \hline 
    $l:$ & 
    12 & 11 & 10 & 9 & 8 & 7 & 6 & 5 & 4 & 3 & 2 & 1 & 0 
    \\
    \hline
    & 
    \multicolumn{12}{c}{$e_{5,n}^{3,l}$ with $l$ running from the right to the left} & 
    \\
    \hline
  \end{tabular}
  \medskip
  \caption{Polar multiplicities for $5\times n$-matrices of rank $2$ and $3$; 
    all entries for $k$ and $l$ which are not explicitly listed, are equal to zero.}
  \label{tab:PolarMultiplicities5byNRank2And3}
\end{table}
  
\end{landscape}
\begin{table}
  \centering
  \begin{tabular}{|r|ccccccccc|}
    \hline 
    & 
    \multicolumn{8}{c}{$e_{5,n}^{1,k}$ with $k$ running from the left to the right} & 
    \\
    \hline
    $k:$ & 
    0 & 1 & 2 & 3 & 4 & 5 & 6 & 7 & 8 
    \\
    \hline 
$5 \times 5$ & 
70& 280& 520& 580& 430& 220& 80& 20& 5 \\
$5 \times 6$ & 
126& 560& 1155& 1440& 1200& 696& 285& 80& 15 \\
$5 \times 7$ & 
210& 1008& 2240& 3010& 2700& 1680& 728& 210& 35 \\
$5 \times 8$ & 
330& 1680& 3948& 5600& 5285& 3440& 1540& 448& 70\\
$5 \times 9$ & 
495& 2640& 6480& 9576& 9380& 6300& 2880& 840& 126 \\
$5 \times 10$ & 
715& 3960& 10065& 15360& 15480& 10640& 4935& 1440& 210 \\
$5 \times 11$ & 
1001& 5720& 14960& 23430& 24150& 16896& 7920& 2310& 330 \\
$5 \times 12$ & 
1365& 8008& 21450& 34320& 36025& 25560& 12078& 3520& 495 \\
$5 \times 13$ & 
1820& 10920& 29848& 48620& 51810& 37180& 17680& 5148& 715 \\
$5 \times 14$ & 
2380& 14560& 40495& 66976& 72280& 52360& 25025& 7280& 1001 \\
$5 \times 15$ & 
3060& 19040& 53760& 90090& 98280& 71760& 34440& 10010& 1365 \\
$5 \times 16$ & 
3876& 24480& 70040& 118720& 130725& 96096& 46280& 13440& 1820 \\
$5 \times 17$ & 
4845& 31008& 89760& 153680& 170600& 126140& 60928& 17680& 2380 \\
$5 \times 18$ & 
5985& 38760& 113373& 195840& 218960& 162720& 78795& 22848& 3060\\
$5 \times 19$ & 
7315& 47880& 141360& 246126& 276930& 206720& 100320& 29070& 3876\\
$5 \times 20$ & 
8855& 58520& 174230& 305520& 345705& 259080& 125970& 36480& 4845\\
\hline
    $l:$ & 
    8 & 7 & 6 & 5 & 4 & 3 & 2 & 1 & 0 
    \\
    \hline
    & 
    \multicolumn{8}{c}{$e_{5,n}^{4,l}$ with $l$ running from the right to the left} & 
    \\
    \hline
  \end{tabular}
  \medskip
  \caption{Polar multiplicities for $5\times n$-matrices of rank $1$ and $4$;
   all entries for $k$ and $l$ which are not explicitly listed, are equal to zero.}
  \label{tab:PolarMultiplicities5byNRank1and4}
\end{table}

\begin{example}
  We may use the above tables together with formula 
  (\ref{eqn:FormulaForEulerCharacteristic}) to compute the Euler characteristics 
  of complex links of higher codimension for the generic determinantal 
  varieties. For instance 
  \[
    \chi\left( \mathcal L^6(M_{3,4}^3,0) \right) = 
    e_{3,4}^{2,0}
    -e_{3,4}^{2,1}
    +e_{3,4}^{2,2}
    -e_{3,4}^{2,3}
    = 6 - 16 + 27 - 24 = -7.
  \]
  Note that since $\mathcal L^6(M_{3,4}^3,0)$ is smooth of complex dimension $3$, 
  the summation over $\alpha \in A$ 
  degenerates and only the smooth stratum $V_{3,4}^2$ is relevant. Moreover, 
  the complex link of $M_{3,4}^3$ along this stratum is empty, so that the factor 
  $(1 - \chi(\mathcal L(M_{3,4}^3,V_{3,4}^2)))$ simply reduces to $1$. 

  This computation confirms the results in an earlier paper \cite{FKZ15} 
  where it was shown that the Betti numbers of $\mathcal L^6(M_{3,4}^3,0)$ are 
  \[
    (b_0,b_1,b_2,b_3) = (1,0,1,9).
  \]
  In \cite{FKZ15}, this was a very particular example. 
  We will discuss in Section \ref{sec:BettiNumbersOfSmoothLinks} how the distinct 
  Betti numbers can be computed for \textit{all} smooth complex links of 
  generic determinantal varieties. 

  \medskip

  To also give an example for a singular complex link, consider 
  $\mathcal L^5(M_{3,4}^3,0)$: This space is of complex dimension $4$ 
  and has isolated singularities 
  which are themselves determinantal of the form $(M_{2,3}^2,0)$. 
  If $D'_6$ is a plane of codimension $6$ in $\CC^{3\times 4}$ in 
  general position off the origin
  such that $\mathcal L^5(M_{3,4}^3,0) = D'_6 \cap M_{3,4}^3$, then 
  these singular points are precisely the intersection points 
  of $D'_6$ with $M_{3,4}^2$ and their number is equal to the multiplicity 
  $e_{3,4}^{1,0} = 10$. 
  
  If we let $l$ be a further, sufficiently general linear form on 
  $\CC^{3\times 4}$, then the generic fiber of its restriction to 
  $M_{3,4}^3 \cap D'_6$ is the previous space $\mathcal L^6(M_{3,4}^3,0)$ 
  whose topology we already know. According to Table 
  (\ref{tab:PolarMultiplicities3byNRank1}), $l$ has $10$ classical 
  Morse critical points on $V_{3,4}^2\cap D'_6$ and $10$ further stratified 
  Morse critical points on $V_{3,4}^1 \cap D'_6 = M_{3,4}^2 \cap D'_6$. 

  For the first set of points, $10$ more cells of real dimension $4$ 
  are added which changes the Euler characteristic by 
  $(-1)^{10-6}\cdot 10 \cdot 1 = 10$ in Formula (\ref{eqn:InductionFormulaForEulerCharacteristics}) 
  (resp. (\ref{eqn:FormulaForEulerCharacteristic})).

  The second set of critical points on the lower dimensional 
  stratum $V_{3,4}^1$ have a nontrivial complex link 
  $\mathcal L(M_{3,4}^3,V_{3,4}^1) \cong \mathcal L^0(M_{2,3}^2,0)$ appearing in the 
  normal Morse datum. This complex link is nothing but the 
  Milnor fiber of the $A_0^+$-singularity in 
  \cite{FKZ15}: Despite being a space of complex dimension $3$,
  it is homotopy equivalent to a 
  $2$-sphere and its Betti numbers are $(b_0,b_1,b_2,b_3) = (1,0,1,0)$ in 
  accordance with the Formula by Ebeling and Gusein-Zade 
  (\ref{eqn:EbelingGuseinZadeFormulaEulerCharacteristic}). This means 
  that we attach real $3$-cells rather than $4$-cells and the Euler characteristic 
  changes by $-10$ rather than $+10$, as one might have expected. The overall outcome therefore is 
  \[
    \chi\left( \mathcal L^5(M_{3,4}^3,0 )\right) = 
    \chi(\mathcal L^6(M_{3,4}^3,0)) 
    + \underbrace{10 \cdot (1-(1-0+1))}_{\alpha = 1}
    + \underbrace{10 \cdot (1-0)}_{\alpha = 2} = -7.
  \]
  It is interesting to see the cancellation of the two contributions to 
  the Euler characteristic
  given that the equality of the two relevant multiplicities is not a 
  coincidence, but due to the duality noted by Gaffney.

  \medskip
  We shall 
  see later on 
  that the first four Betti numbers of the open stratum $V_{3,4}^2 \cap D'_6$ 
  of $\mathcal L^5(M_{3,4}^3,0)$ are 
  \[
    (b_0,b_1,b_2,b_3) = (1,0,1,0).
  \]
  It can be shown that the attachements of the $3$-cells at the points of 
  $V_{3,4}^1 \cap D'_6$ glue their boundaries all to the very same 
  generator of the second homology group. From the long exact 
  sequence of the pair $(X_{3,4}^2\cap D'_6, V_{3,4}^2 \cap D'_6)$ 
  and the previous computation 
  of the Euler characteristic 
  one can then deduce 
  that the Betti numbers of 
  $\mathcal L^5(M_{3,4}^3,0)$ must be 
  \[
    (b_0,b_1,b_2,b_3,b_4) = (1,0,0,9,1).
  \]
  In particular we see that the cells attached at the 
  classical critical points of $l$ on the smooth stratum 
  kill off all the cycles in the top homology group of 
  $\mathcal L^6(M_{3,4}^3,0)$. Those coming from the stratified 
  Morse critical points on the lower dimensional stratum survive 
  and lead to new cycles.
  Details for the computation of the Betti numbers 
  in this example will appear in a forthcoming note. 
\end{example}

\section{Determinantal strata as homogeneous spaces}

Let $G$ be a Lie group and $* \colon G \times X \to X$ a smooth 
action on a manifold $X$. Then for every point $x\in X$ the orbit 
$G*x \subset X$ is a locally closed submanifold which is diffeomorphic 
to the quotient $G/G_x$ of $G$ by the stabilizer $G_x$ of $x$. 
The next lemma shows that up to homotopy we can 
always find a \textit{compact model} for this orbit by choosing 
an appropriate maximal compact subgroup of $G$.

\begin{lemma}
  Let $G$ be a Lie group, $G' \subset G$ a closed subgroup, and 
  $U \subset G$ a maximal compact subgroup such that $U' = U \cap G'$ is 
  again a maximal compact subgroup of $G'$. Then the inclusion 
  $U/U' \hookrightarrow G / G'$ is a weak homotopy equivalence.
  \label{lem:HomotopyEquivalenceForCompactSubgroups}
\end{lemma}

\begin{proof}
  The projection $G \to G/G'$ is a fiber bundle with fiber 
  $G'$ and the same holds for $U \mapsto U/U'$ 
  with fiber $U'$.
  Hence, there is a commutative diagram of long exact sequences of homotopy groups 
  \[
    \xymatrix{
      \cdots \ar[r] & 
      \pi_k( G' ) \ar[r] & 
      \pi_k( G ) \ar[r] & 
      \pi_k( G/G' ) \ar[r] & 
      \pi_{k-1}(G' ) \ar[r] & 
      \cdots\\
      \cdots \ar[r] & 
      \pi_k( U' ) \ar[u] \ar[r] & 
      \pi_k( U ) \ar[u] \ar[r] \ar[u] & 
      \pi_k( U/U' ) \ar[r] \ar[u] & 
      \pi_{k-1}( U' ) \ar[r] \ar[u] & 
      \cdots
    }
  \]
  and it is well known that for any Lie group the inclusion of its maximal compact 
  subgroup is a homotopy equivalence. The assertion therefore follows from the five-lemma.
\end{proof}

\subsection{The Lie group action on $\CC^{m \times n}$}
We now turn to the discussion of the strata in the rank stratification 
as homogeneous spaces. Fix integers $0 < m \leq n$. The space $\CC^{m\times n}$ of complex 
$m\times n$-matrices has a natural left action by the complex Lie group 
\[
  G_{m,n} := \GL(m;\CC) \times \GL(n;\CC)
\]
via multiplication:
\[
  * \colon G_{m,n} \times \CC^{m\times n} \to \CC^{m\times n}, 
  \quad 
  ((P,Q),A) \mapsto (P,Q)*A = P \cdot A \cdot Q^{-1}.
\]
For two matrices $A \in \CC^{m\times n}$ and $B \in \CC^{m' \times n'}$ 
we will denote by $A \oplus B$
the $(m+m') \times (n+n')$ block matrix 
\[
  \begin{pmatrix}
    A & 0 \\
    0 & B
  \end{pmatrix}.
\]
Let $0_{m,n} \in \CC^{m\times n}$ be the zero matrix. 
For any number $0\leq r \leq m$ 
we will write $\one_{m,n}^r = \one^r \oplus 0_{m-r,n-r}$ for 
the $m\times n$-matrix with a unit matrix $\one^r$
of rank $r$ in the upper left corner and zeroes in all other entries. 
Then clearly 
\[
  V_{m,n}^r = G * \one_{m,n}^r \cong G_{m,n}/G_{m,n}^r,
\]
where $G_{m,n}^r$ is the stabilizer of $\one_{m,n}^r$ in $G$. 
A direct computation yields that $G_{m,n}^r$ consists of pairs 
of block matrices of the form 
\[
  \left( 
  \begin{pmatrix}
    A & B \\
    0 & C 
  \end{pmatrix},
  \begin{pmatrix}
    A & 0 \\ 
    D & E
  \end{pmatrix}
  \right)
\]
with $A \in \GL(r,\CC)$, $C \in \GL(m-r;\CC)$ and $E\in \GL(n-r;\CC)$ invertible, 
and $B$ and $D$ arbitrary of appropriate sizes. 

As a compact subgroup $U_{m,n} \subset G_{m,n}$ we may choose the 
unitary matrices $U(m) \times U(n)$. 
It is easily verified that its intersection
$U_{m,n}^r$ with 
the subgroup $G_{m,n}^r$ consists of pairs of matrices 
\[
  \left( 
  \begin{pmatrix}
    S & 0 \\
    0 & P 
  \end{pmatrix},
  \begin{pmatrix}
    S & 0 \\
    0 & Q
  \end{pmatrix}
  \right)
  = \left( S \oplus P, S \oplus Q \right)
\]
with $S \in U(r)$, $P \in U(m-r)$, and $Q \in U(n-r)$, 
and that this is in fact a maximal compact subgroup of $G_{m,n}^r$.
Note that due to the fact that in contrast to $G_{m,n}^r$ the off-diagonal blocks in the 
subgroup $U_{m,n}^r$ are all zero, we find that 
\begin{equation}
  U_{m,n}^r \cong U(r) \times U(m-r) \times U(n-r)
  \label{eqn:SplittingOfStabilizerOneOrbit}
\end{equation}
is again isomorphic to a product of unitary groups. 

\medskip

Let us, for the moment, consider only 
the first factor $U(m)$ of $U_{m,n}$ which 
we may consider as a subgroup via the inclusion $U(m) \times \{\one^r\} \subset U_{m,n}$. 
The stabilizer 
of $\one_{m,n}^r$ of the restriction of the action to 
$U(m)$ is simply the subgroup $\one^r \oplus U(m-r)$. 
The $U(m)$-orbit can easily be identified with the 
Stiefel manifold $\Stief(r,m)$ of orthonormal 
$r$-frames in $\CC^m$:
\begin{equation}
  \Stief(r,m) \cong U(m)/\left( \one^r \oplus U(m-r) \right) 
  \cong U(m)* \one_{m,n}^r
  \label{eqn:StiefelManifoldAsHomogeneousSpace}
\end{equation}
so that an $r$-frame $\underline v = (v_1,\dots,v_r)$ in $\CC^m$ is given by the 
first $r$ columns of the matrix $(\underline v) = A \cdot \one_{m,n}^r$ for 
some $A\in U(m)$.
The group $U(r)$ operates naturally on the Stiefel 
manifold via the left action
\[
  U(r) \times \Stief(r,m) \to \Stief(r,m), \quad 
  (S, A\cdot \one^r_{m,r} ) 
  \mapsto A \cdot \one^r_{m,r} \cdot S^{-1}.
\]
The quotient of this action is the Grassmannian 
of $r$-planes 
$\Grass(r,m)$ since either two $r$-frames span 
the same subspace if and only if they lay in the 
same orbit under this $U(r)$-action.

It is easy to see with the 
above identifications (\ref{eqn:StiefelManifoldAsHomogeneousSpace}) 
that two matrices $A$ and $A'$ in $U(m)$ represent the 
same element in $\Grass(r,m)$ if and only if 
$A^{-1} \cdot A' \in U(r) \oplus U(m-r)$: 
\begin{eqnarray*}
  A \cdot \one_{m,r}^r &=& A' \cdot \one_{m,r}^r \cdot S^{-1} \\
  \Leftrightarrow \qquad \one_{m,r}^r &=& 
  A^{-1} \cdot A' \cdot (S^{-1} \oplus \one^{m-r}) \cdot \one_{m,r}^r\\
  \Leftrightarrow \quad 
  ( \one^r \oplus P ) 
  &=& 
  A^{-1} \cdot A' \cdot (S^{-1} \oplus \one^{m-r} ) \\
  \Leftrightarrow \quad
  (S \oplus P)
  &=&  
  A^{-1} \cdot A' 
\end{eqnarray*}
for some $P \in U(m-r)$.
In other words, the above $U(r)$ action is compatible 
with the natural inclusion of subgroups
\[
  U(r) \oplus \one^{m-r} \hookrightarrow 
  U(r) \oplus U(m-r) \hookrightarrow U(m)
\]
and accordingly 
\begin{equation}
  \Grass(r,m) \cong \Stief(r,m)/U(r) \cong U(m)/U(r) \oplus U(m-r).
  \label{eqn:GrassmannianAsHomogeneousSpace}
\end{equation}

We can repeat these considerations for the second factor $U(n)$ 
embedded into $U_{m,n}$ as $\{\one^m\} \times U(n)$.
Then 
\[
  \Stief(r,n) \cong U(n)/(\one^r\oplus U(n-r)) \cong U(n) * \one_{m,n}^r
\]
with any $r$-frame $\underline w \in \Stief(r,n)$ 
given by the first $r$ \textit{rows} of the matrix $(\underline w) = \one^{r}_{m,n} \cdot B^{-1}$ 
for some $B \in U(n)$. Accordingly, we will write the left action by $U(r)$ 
on $\Stief(r,n)$ as 
\[
  U(r) \times \Stief(r,n) \to \Stief(r,n), \quad 
  (S, \one_{m,n}^r \cdot B^{-1} ) \mapsto S \cdot \one_{m,n}^r \cdot B^{-1}
\]
in this case. 

Note that, on the one hand, the subgroup 
$U_{m,n}^r$ intersects the subgroups $U(m) \times \{\one^n\}$ and 
$\{\one^m\}\times U(n)$ in $\one^r \oplus U(m-r)$ and  
$\one^r \oplus U(n-r)$, respectively, and  
the action of the latter subgroups affects only either one of the two 
factors.
The $U(r)$-action, on the other hand, is ``diagonal''
and we may exploit these facts by observing 
that the quotient 
\begin{equation}
  U_{m,n}/U(m-r)\times U(n-r) \cong \Stief(r,m) \times \Stief(r,n)
  \label{eqn:ProductOfStiefelVarieties}
\end{equation}
is a product of Stiefel manifolds, equipped with a \textit{free}, diagonal 
$U(r)$-action. 
The quotient $\Stief(r,m)\times \Stief(r,n)/U(r)$ is then naturally isomorphic to 
$U_{m,n}/\left( U(r) \times U(m-r) \times U(n-r) \right)$ and, via the particular 
choice of the matrix $\one_{m,n}^r$, this manifold can be identified with the orbit 
$U_{m,n}*\one_{m,n}^r$. 
We will in the following denote this orbit by $O_{m,n}^r \subset V_{m,n}^r \subset \CC^{m\times n}$ 
and refer to it as the \textit{compact orbit model} for the stratum $V_{m,n}^r$. 

\begin{lemma}
  The two natural projections 
  \begin{equation*}
    \xymatrix{
      \Stief(r,n) \ar@{^{(}->}[dr] & & 
      \Stief(r,m) \ar@{_{(}->}[dl] \\
      &O_{m,n}^r\ar[dl]_{\lambda_2} \ar[dr]^{\lambda_1} & \\
      \Grass(r,n) & & \Grass(r,m)
    }
  \end{equation*}
  equip the space $U_{m,n}/U_{m,n}^r \cong U_{m,n}*\one_{m,n}^r$ with two structures as a fiber bundle 
  over the respective Grassmannian with Stiefel manifolds as fibers.
  \label{lem:FiberBundleStructureOfROrbitOverGrassmannian}
\end{lemma}
\begin{proof}
  It suffices to establish this claim for the first projection $\lambda_1$ to $\Grass(r,m)$.
  Consider the commutative diagram 
  \[
    \xymatrix{
      \Stief(r,m)\times \Stief(r,n) \ar[d]^{\lambda'_1} \ar[r]^-{\alpha} & 
      O_{m,n}^r \ar[d]^{\lambda_1} \\
      \Stief(r,m) \ar[r]^{\beta} & 
      \Grass(r,m)
    }
  \]
  where $\lambda'_1$ takes a pair of $r$-frames $(\underline v, \underline w)$ to $\underline v$, 
  $\beta$ is the quotient map $\underline v \mapsto \Span \underline v$ 
  from (\ref{eqn:GrassmannianAsHomogeneousSpace}) and 
  $\alpha$ the one from the discussion of (\ref{eqn:ProductOfStiefelVarieties}). 
  We need to describe the fiber of an arbitrary point $W \in \Grass(r,m)$. To this 
  end, consider its preimages
  \[
    (\beta \circ \lambda'_1)^{-1}(\{W\}) =
    \beta^{-1}(\{W\}) \times \Stief(r,n) 
    \overset{\lambda'_1}{\longrightarrow} 
    \beta^{-1}(\{W\}) = U(r) * \underline v
  \]
  with $\underline v = (v_1,\dots,v_r) \in \Stief(r,m)$ some $r$-frame in $\CC^m$ 
  with $\Span \underline v = W$. The fiber of $\lambda'_1$ over $\underline v$ is 
  simply the Stiefel manifold $\Stief(r,n)$. Now if $\underline v'$ is any other 
  $r$-frame spanning $W$, then there exists a unique matrix $S\in U(r)$ such that 
  $\underline v' = \underline v \cdot S^{-1}$. The free diagonal action 
  on $\Stief(r,m) \times \Stief(r,n)$ gives a natural identification of the fibers 
  of $\lambda'_1$ over $\underline v$ and $\underline v'$ via 
  \[
    (\underline v, \underline w) \mapsto (\underline v', S \cdot \underline w)
  \]
  and this furnishes an obvious notion of \textit{parallel sections} of $\lambda'_1$ 
  over the orbit $U(r) * \underline v$. These parallel sections can then 
  be identified with either one 
  of the Stiefel manifolds $\Stief(r,n) \times \{\underline v\}$ 
  over a point $\underline v$ in the orbit.
\end{proof}

\begin{remark}
  The structures of the manifolds $O_{m,n}^r$ as fiber bundle is in general not 
  trivial. For instance 
  \[
    S^{2m-1} \cong O_{m,1}^1 \overset{\lambda_1}{\longrightarrow} \Grass(1,m) \cong \PP^{m-1}
  \]
  is the Hopf fibration which is known not to be a product.
\end{remark}

\subsection{The Cartan model}

The cohomology of homogeneous spaces can be computed via the Cartan model 
as outlined by Borel in \cite[Th\'eor\`eme 25.2]{Borel53}. 
Let $G$ be a \textit{compact}, connected Lie group and $U \subset G$ a closed subgroup thereof. 
Then \cite[Th\'eor\`eme 25.2]{Borel53} allows for the description of the cohomology 
of the quotient $G/U$ as the cohomology of an explicit complex 
\begin{equation}
  H^\bullet( G/U ) \cong H\left( S_{U} \otimes_\ZZ \bigwedge F \right)
  \label{eqn:BorelComplex}
\end{equation}
under certain favourable assumptions on $G$ and $U$. The objects on the 
right hand side are the following. 

\begin{itemize}
  \item The ring $S_U$ is the cohomology ring of a \textit{classifying space} for 
    the group $U$, see for instance \cite[Part I, Chapter 4.11]{Husemoeller94}. 
    Such a classifying space $BU$ for a compact Lie group $U$ is given by the 
    quotient of any weakly contractible space $EU$ with a free $U$-action. Then 
    the projection $EU \to BU$ turns $EU$ into a \textit{universal bundle} in the 
    sense that every principal $U$-bundle $P$ over a paracompact Haussdorff space 
    $X$ can be written as $P = f^* EU$ for some continuous map $f \colon X \to BU$. 

    In particular, this property can be used to show that the cohomology ring 
    $S_U = H^\bullet(BU)$ is in fact unique up to unique isomorphism and independent 
    of the choice of the space $EU$, see e.g. 
    \cite[Section 18]{Borel53}\footnote{The approach by Borel might seem unnecessarily technical 
      given the Milnor construction of universal bundles in \cite{Milnor56} three years later.}. 
    In most practical cases (cf. 
    \cite[Th\'eor\`eme 19.1]{Borel53}) the cohomology rings of 
    classifying spaces are weighted homogeneous 
    polynomial rings $S_U \cong \ZZ[c_1,\dots,c_r]$ in variables of even degree 
    and therefore in particular commutative.

    Furthermore, we note that the total space $EG$ of a universal $G$ bundle 
    is naturally equipped with a free $U$ action, as well. It can therefore 
    also be used to construct a classifying space $BU$ as the intermediate quotient 
    \[
      EG \to BU = EG/U \to BG = EG/G.
    \]
    The pullback in cohomology of the projection $BU \to BG$ is called the 
    \textit{characteristic homomorphism} $\rho \colon S_G \to S_U$ for the 
    inclusion of the subgroup $U \subset G$, cf. \cite[Th\'eor\`eme 22.2]{Borel53}.
  \item The module $F$ is a free, graded $\ZZ$-module 
    \[
      F = \ZZ \varepsilon_1 \oplus \dots \oplus \ZZ \varepsilon_r
    \]
    in generators $\varepsilon_i$ of odd degree. Hopf has shown 
    in \cite[Satz 1]{Hopf41} that 
    the rational homology of a compact Lie group $U$ is graded isomorphic 
    to the homology of a product of odd-dimensional spheres. 
    This result has been strengthened
    to also include cohomology with integer coefficients in the 
    absence of torsion in $H^\bullet(U)$, see e.g. \cite[Proposition 7.3]{Borel53}. 
    Then the generators $\varepsilon_i$ can be thought of as the volume 
    forms of the spheres and the cup product turns the cohomology of the 
    group into an exterior algebra on these generators
    \[
      H^\bullet(U) \cong 
      \bigwedge \left( \ZZ \varepsilon_1 \oplus \dots \oplus \ZZ \varepsilon_r \right)
    \]
    which appears as $\bigwedge F$ in (\ref{eqn:BorelComplex}).
  \item The differential $D$ on $S_U \otimes_\ZZ \bigwedge F$ is given by 
    linear extension of the map
    \[
      D ( a \otimes 1 ) = 0, \quad D( 1 \otimes \varepsilon_i ) = \rho(c_i) \otimes 1
    \]
    for all $a\in S_U$, where $c_i$ is a \textit{transgression} element of $\varepsilon_i$ in a 
    universal $G$-bundle and $\rho \colon S_G \to S_U$ the characteristic 
    homomorphism from before. A transgression can be defined in a universal $U$-bundle 
    $\pi \colon EU \to BU$ 
    as above: The element $\varepsilon_i$ is called \textit{universally transgressive} 
    if there exists a cochain $\omega_i$ on $EU$ which restricts to the cohomology class 
    $\varepsilon_i$ in every fiber and for which there exists another cochain $a_i$ on 
    $BG$ with $\pi^* a_i = \D \omega_i$. Then $c_i$ is taken to be the cohomology class of $a_i$ 
    and it is said to ``correspond to $\varepsilon_i$ under transgression''. 
    For a more detailed account see \cite{Borel53}.
    We will discuss the particular form of this 
    transgression below in the cases of interest for this article. 
\end{itemize}
Note that (\ref{eqn:BorelComplex}) is an isomorphism of graded $\ZZ$-modules 
with grading given by the degree of the cohomology classes on either side. 
But we can also think of the algebra $S_U \otimes_{\ZZ} \bigwedge F$ as a 
Koszul algebra in the generators $1 \otimes_\ZZ \varepsilon_i$ over the ring $S_U$:
\[
  S_U \otimes_\ZZ \bigwedge F \cong \bigoplus_{p=0}^r 
  \left(\bigwedge^p \left(\bigoplus_{i=1}^r S_U \otimes_{\ZZ} \varepsilon_i \right)\right).
\]
This gives another grading on the right hand side of (\ref{eqn:BorelComplex})
by the degrees $p$ of the distinct exterior powers. 
In the following we will refer to the two gradings as the \textit{cohomological degree} 
and the \textit{Koszul degree} respectively.

\subsection{Maximal tori and the Weyl group}

As explained in \cite[Section 29]{Borel53}, the characteristic homomorphism 
$\rho \colon S_G \to S_U$ associated to an inclusion of a subgroup $U \subset G$ 
is best understood in terms of inclusions of maximal tori $S \subset U$ 
and $T \subset G$. This will be an essential ingredient for the computation of the cohomology 
of the orbits $O_{m,n}^r$ in the determinantal strata.

For the group $U(1) \cong S^1 \subset \CC$ a classifying 
space $BU(1)$ is given by the infinite projective space 
\[
  BU(1) \cong \bigcup_{k=1}^\infty \PP^k
\]
which can be understood as a direct limit with $\PP^k$ included in 
$\PP^{k+1}$ as the hyperplane section at infinity. A universal 
$U(1)$-bundle is then given by $\OO(-1)^*$, the tautological bundle 
with its zero section removed or, equivalently, by the direct limit 
of unit spheres $S^{2k+1} \subset \CC^{k+1}\setminus\{0\}$ which 
are projected to $\PP^k$ via the Hopf fibration. Then the 
cohomology ring $S_{U(1)} \cong \ZZ[\alpha]$ of $BU(1)$ is a free polynomial 
ring 
in $\alpha$, the first Chern class of $\OO(1)$, and the 
generator $\varepsilon$ of $H^1( U(1))$ corresponds to $\alpha$ 
under transgression. 

Now it is easy to see that for a torus $T = (U(1))^r$ the classifying 
spaces and universal bundles can be chosen to be merely products of 
the one just described for $U(1)$. It follows that $S_T$ is 
a polynomial ring 
\[
  S_T \cong \ZZ[\alpha_1,\dots,\alpha_r]
\]
with all 
$\alpha_j$ of degree $2$. 

If $T$ is a maximal torus in a compact 
connected Lie group $G$, then $S_G$ is contained in $S_T$ as 
the invariant ring under the action of the Weyl group. 
Moreover, if $U\subset G$ is a subgroup and the tori $S\subset U$ 
and $T\subset G$ have been chosen such that $S \subset T$, 
then the characteristic homomorphism $\rho = \rho(U,G)$ 
for the inclusion $U \subset G$ 
is completely determined by the one for the inclusion $S \subset T$ 
so that one has a commutative diagram 
\begin{equation}
  \xymatrix{
    S_G \ar[r]^{\rho(U,G)} \ar@{^{(}->}[d] &
    S_U \ar@{^{(}->}[d] \\
    S_T \ar[r]^{\rho(S,T)} & 
    S_S
  }
  \label{eqn:InducedMapOnInvariantSubrings}
\end{equation}
where the vertical arrows denote the inclusions of invariant subgroups. 
We note that in general one needs real coefficients in cohomology 
as in \cite[Proposition 29.2]{Borel53}. For the particular cases 
that we shall need below, however, the calculations have been 
carried out for integer coefficients as well.

\subsection{Cohomology of Stiefel manifolds and Grassmannians}
\label{sec:StiefelManifoldsAndFlagVarieties}

As discussed earlier the Stiefel manifolds and Grassmannians  
can be considered as homogeneous spaces of $U(n)$ modulo 
various subgroups of block matrices with unitary blocks. Also 
the orbit varieties $O_{m,n}^r$ can be decomposed 
into these building blocks. Therefore we briefly review the 
theory for the Lie group $U(n)$ as it can be found in 
\cite{Borel53} or \cite{Husemoeller94} and illustrate the 
formula (\ref{eqn:BorelComplex}) for the classical cases, 
thereby fixing notation for the description of the cohomology 
of the orbit models $O_{m,n}^r$ that we are 
really aiming for. 

\medskip

The cohomology of the unitary group 
$U(n)$ is known to be 
\begin{equation}
  H^\bullet(U(n)) \cong \bigwedge \left( 
  \ZZ\varepsilon_1 \oplus
  \ZZ\varepsilon_2 \oplus
  \dots 
  \ZZ \varepsilon_{n-1} \oplus
  \ZZ \varepsilon_n 
  \right)
  \label{eqn:CohomologyUN}
\end{equation}
with generators $\varepsilon_i$ of degree $2i-1$, see 
for example \cite[Proposition 9.1]{Borel53} or 
\cite[Part I, Chapter 7]{Husemoeller94}. 
We will write 
\[
  F_n = 
  \ZZ[-1] \oplus 
  \ZZ[-3] \oplus 
  \dots \oplus 
  \ZZ[3-2n] \oplus 
  \ZZ[1-2n] 
\]
for the free graded module whose direct summands are shifted by 
$1-2i$ for $1 \leq i \leq n$ so that with this notation 
$H^\bullet(U(n)) \cong \bigwedge F_n$. 
A maximal torus in $U(n)$ is given by the diagonal matrices
\begin{eqnarray*}
  T &=&  \{\mathrm{diag}(\lambda_1,\dots,\lambda_n) : \lambda_i \in U(1) \}
\end{eqnarray*}
and the Weyl group is the symmetric group $\mathfrak S_n$ of permutations 
of $n$ elements in this case. Writing $\alpha_1,\dots,\alpha_n$ for the 
generators of the cohomology of $S_T$ as above we find that 
\[
  S_G \hookrightarrow S_T
\]
is the inclusion of the invariant subring $S_G = \ZZ[\alpha_1,\dots,\alpha_n]^{\mathfrak S_n}$. 
According to the fundamental theorem of symmetric functions, $S_G$ is itself a 
polynomial ring in the elementary symmetric functions 
\[
  \sigma_k^n(\alpha_1,\dots,\alpha_n) := \sum_{0<i_1<\dots<i_k\leq n} \prod_{j=1}^n \alpha_{i_j}.
\]
Moreover, the generators $\varepsilon_k$ correspond to these $\sigma_k$ under 
transgression in a universal bundle $EU(n) \to BU(n)$, see for example 
\cite[Section 19]{Borel53}. 

\medskip

The cohomology of Stiefel manifolds $\Stief(r,n)$ turns out to be a truncated version 
of the cohomology of $U(n)$:
\begin{equation}
  H^\bullet( \Stief(r,n)) \cong \bigwedge F_{r}[2r-2n]. 
  \label{eqn:CohomologyStiefelManifold}
\end{equation}
In order to make the connection with Formula (\ref{eqn:BorelComplex}) recall 
that $\Stief(r,n) \cong U(n)/(\one^r \oplus U(n-r))$. As maximal tori in 
the subgroup $U' := \one^r \oplus U(n-r)$ we may choose 
\begin{eqnarray*}
  T' &=&  \{ \one^r \oplus \mathrm{diag}(\mu_{1},\dots,\mu_{n-r}) : \mu_j \in U(1) \}
\end{eqnarray*}
so that $T'\subset T$ with $T \subset U(n)$ as before. Writing 
$S_{T'} = \ZZ[\beta_1,\dots,\beta_{n-r}]$ for the cohomology of the classifying space 
we find that the diagram (\ref{eqn:InducedMapOnInvariantSubrings}) becomes 
\[
  \xymatrix{
    \ZZ[\alpha_1,\dots,\alpha_n]^{\mathfrak S_n} \ar[r]^-{\rho(U',U)} \ar@{^{(}->}[d] &
    \ZZ[\beta_1,\dots,\beta_{n-r}]^{\mathfrak S_{n-r}} \ar@{^{(}->}[d] \\
    \ZZ[\alpha_1,\dots,\alpha_n] \ar[r]^-{\rho(T',T)} & 
    \ZZ[\beta_1,\dots,\beta_{n-r}]
  }
\]
where the map $\rho(T',T)$ is given by 
\[
  \rho(T',T) \colon \alpha_j \mapsto
  \begin{cases}
    \beta_{j-r} & \textnormal{ if } j>r, \\
    0 & \textnormal{ otherwise. }
  \end{cases}
\]
It follows that $\rho(U',U)$ is merely a substitution of the 
variables in the symmetric functions given by $\rho(T',T)$ so that 
\[
  \rho(U',U) \colon \sigma_k^n(\alpha_1,\dots,\alpha_n) \mapsto
  \begin{cases}
    \sigma_k^{n-r}(\beta_1,\dots,\beta_{n-r}) & \textnormal{ if } k \leq n-r, \\
    0 & \textnormal{ otherwise.}
  \end{cases}
\]
With these considerations at hand we can now investigate the 
homology of the complex $S_{U'} \otimes_\ZZ \bigwedge F_n$, 
with its transgression differential $D$ from 
(\ref{eqn:BorelComplex}). To this end 
we shall apply the following well known reduction lemma.

\begin{lemma}
  Let $R$ be a ring, $M$ an $R$-module and $x,y_1,\dots,y_n \in R$ 
  be elements. When $x$ is a non-zerodivisor on $M$, then there 
  is a canonical isomorphism 
  \[
    H( \Kosz(x,y_1,\dots,y_n;M)) \cong H( \Kosz(y_1,\dots,y_n;M/xM)
  \]
  for the entire cohomology of the Koszul complexes.
  \label{lem:KoszulComplexAndRegularElements}
\end{lemma}

\begin{proof}
  See \cite[Corollary 1.6.13 (b)]{BrunsHerzog93}.
\end{proof}

The ring $S_{U'} \cong \ZZ[c_1,\dots,c_{n-r}]$ is 
freely generated in the elementary symmetric functions in 
the $\beta_1,\dots,\beta_{n-r}$. 
Now the transgression differential $D$ takes the form 
\[
  D \colon 1 \otimes \varepsilon_i \mapsto
  \begin{cases}
    c_i & \textnormal{ for } i \leq n-r \\
    0 & \textnormal{ otherwise. }
  \end{cases}
\]
Since the $c_i$ form a regular sequence on the module $M = S_{U'}$ 
with quotient $S_{U'}/\langle c_1,\dots,c_{n-r} \rangle \cong \ZZ$, 
it follows inductively from Lemma \ref{lem:KoszulComplexAndRegularElements} 
that 
\[
  H(\Stief(r,n)) \cong 
  H\left( S_{U'} \otimes_{\ZZ} \bigwedge F_n \right) \cong 
  H\left( \Kosz(0,\dots,0;\ZZ \right) \cong \bigwedge F_r[2r-2n]
\]
as anticipated.

\medskip

Let $0 \leq r \leq n$ be integers.
The case of a Grassmannian
\[
  \Grass(r,n) \cong U(n) / (U(r) \oplus U(n-r))
\]
is similar, only that the maximal tori for $U(n)$ and its subgroup 
$U' = U(r) \oplus U(n-r)$
can be chosen to be the same so that $S_T = S_{T'} = \ZZ[\alpha_1,\dots,\alpha_n]$. 
The difference comes from the Weyl groups. 
For $U(n)$ we again have the full symmetric group $\mathfrak S_n$, but 
for the subgroup $U'$ we find 
\[
  \mathfrak S_r \times \mathfrak S_{n-r} \subset \mathfrak S_n
\]
whose action on $\ZZ[\alpha_1,\dots,\alpha_n]$ respects the partion 
of the variables $\alpha_j$ into subsets 
$\{ \alpha_1,\dots,\alpha_{r}\}$ and $\{ \alpha_{n-r+1}, \dots,\alpha_{n} \}$.
We write 
\[
  x_j = \sigma^r_j(\alpha_1,\dots,\alpha_r),\quad
  y_k = \sigma^{n-r}_k(\alpha_{n-r+1},\dots,\alpha_n)
\]
for the elementary symmetric polynomials in the respective 
set of variables. Then $S_{U'}$
is a free polynomial subring in these variables
\[
  S_{U'} \cong \ZZ[x_j,y_k : 0< j \leq r, \, 0 < k \leq n-r ] \subset \ZZ[\alpha_1,\dots,\alpha_n]
\]
containing $S_U$ as the invariant subring under \textit{arbitrary} permutations, 
i.e. forgetting 
about the particular partition. It is an elementary exercise in the 
theory of symmetric functions to verify that in this situation 
\begin{equation}
  \sigma_d^n( \underline \alpha ) =
  \sum_{j+k = d} \sigma_j^r(\alpha_1,\dots,\alpha_r) \cdot 
  \sigma_k^{n-r}(\alpha_{n-r+1},\dots,\alpha_n) 
  = \sum_{j+k = d} x_j \cdot y_k.
    \label{eqn:ElementarySymmetricFunctionsAndPartitions}
\end{equation}
Now the complex $S_{U'} \otimes_\ZZ \bigwedge F_n$ in (\ref{eqn:BorelComplex}) looks 
as follows, cf. \cite[Proposition 31.1]{Borel53}:
The differential $D$ takes each one of the generators $\varepsilon_k$ to 
$D(1\otimes \varepsilon_k) = \sigma_k^n(\alpha_1,\dots,\alpha_n)$. 
These are $n$ weighted homogeneous relations in a free graded polynomial ring 
with $n$ variables $x_1,\dots,x_r,y_1,\dots,y_{n-r}$. 
From the fact that the cohomology of the associated 
Koszul complex is finite dimensional, we see that the elements 
(\ref{eqn:ElementarySymmetricFunctionsAndPartitions}) must form a regular 
sequence on $S_{U'}$ so that the 
complex is exact except at Koszul degree zero where we find 
\begin{equation}
  H^\bullet(\Grass(r,n)) \cong S_{U'}/I_{r,n}
  \label{eqn:CohomologyFlagVariety}
\end{equation}
with $I_{r,n}$ the weighted homogeneoeus ideal generated by the elements 
(\ref{eqn:ElementarySymmetricFunctionsAndPartitions}).

\begin{remark}
  \label{rem:CohomologyOfFlagVarietyAsChernClasses}
  This model for the cohomology is linked to the geometry of the Grassmannians 
  as follows. Let 
  \[
    0 \to \mathcal S \to \OO^n \to \mathcal Q \to 0
  \]
  be the tautological sequence on $\Grass(r,n)$. 
  Then modulo $I_n$ we find that 
  $c_j = c_j(\mathcal S)$ is the $j$-th Chern class 
  of the tautological bundle and $s_k = c_k(\mathcal Q)$ 
  the $k$-th Chern class of the tautological quotient bundle. 
  The latter are nothing but the $k$-th Segre classes of $\mathcal S$ 
  which gives us precisely the relations 
  (\ref{eqn:ElementarySymmetricFunctionsAndPartitions}) 
  by expansion of the product of total Chern classes
  \[
    c_t (\mathcal S) \cdot c_t (\mathcal Q) = c(\OO^{n}) = 1
  \]
  in all degrees.
\end{remark}

This cohomological model can be simplified further. 
Let $x_0 = 1, x_1,\dots,x_r$ be the Chern classes of the tautological subbundle and 
$y_0 = 1, y_1,\dots,y_{n-r}$ those of the tautological quotient bundle. 
The relations given by the images $D (1 \otimes \varepsilon_i)$, i.e. 
the generators of the ideal $I_{r,n}$ are 
\[
  \begin{matrix}
    0 = & x_1 & + & y_1 \\
    0 = & x_2 & + & x_1 \cdot y_1 & + & y_2 \\
    \vdots   & \vdots & & \vdots & & & \ddots \\
    0 = & x_r & + & x_{r-1} \cdot y_1 & + & \dots & + & y_r \\
    0 = & x_r \cdot y_1 & + & x_{r-1} \cdot y_2 & + & \dots & + & y_{r+1} \\
    \vdots   & \vdots & & \vdots & & & & \vdots \\
    0 = & x_{r} \cdot y_{n-2r} & + & x_{r-1} \cdot y_{n-2r+1} & + & \dots & + & y_{n-r} \\
    0 = & x_{r} \cdot y_{n-2r-1} & + & x_{r-1} \cdot y_{n-2r} & + & \dots & + & x_1 \cdot y_{n-r} \\
    0 = & x_{r} \cdot y_{n-2r} & + & x_{r-1} \cdot y_{n-2r+1} & + & \dots & + & x_2 \cdot y_{n-r} \\
    \vdots   & & \ddots & & \ddots & & & \vdots \\
    0 = & & & x_r \cdot y_{n-r-2} & + & x_{r-1} \cdot y_{n-r-1} & + & x_{r-2} \cdot y_{n-r} \\
    0 = & & & & & x_r \cdot y_{n-r-1} & + & x_{r-1} \cdot y_{n-r} \\
    0 = & & & & & & & x_r \cdot y_{n-r}.
  \end{matrix}
\]
The first $n-r$ equations can be used to eliminate all of the $y$-variables and express 
them in terms of $x$. Substituting these into the last $r$ equations we obtain polynomials 
in $x$ which we denote by 
\[
  h_1^{(n)}, h_2^{(n)}, \dots, h_r^{(n)}.
\]
If we let $J_{r,n} = \left\langle h_1^{(n)}, h_2^{(n)}, \dots, h_r^{(n)} \right\rangle 
\subset \ZZ[x_1,\dots,x_r]$ 
be the ideal generated by these elements then 
\begin{equation}
  H^\bullet( \Grass(r,n)) \cong \ZZ[x_1,\dots,x_r]/J_{r,n}.
  \label{eqn:SimplifiedGrassmannianCohomology}
\end{equation}
Moreover, the polynomials $h_k^{(n)}$ satisfy a recurrence relation 
that can easily be derived from their explicit construction.
Writing them in a vector we have 
\[
  \begin{pmatrix}
    h_1^{(0)} & h_2^{(0)} & \dots & h_r^{(0)} 
  \end{pmatrix}^T
  = 
  \begin{pmatrix}
    x_1 & x_2 & \dots & x_r
  \end{pmatrix}^T
\]
and 
\begin{equation}
  \begin{pmatrix}
    h_1^{(n+1)} \\
    h_2^{(n+1)} \\
    \vdots \\
    h_{r-1}^{(n+1)} \\
    h_r^{(n+1)} 
  \end{pmatrix}
  = 
  \begin{pmatrix}
    -x_1 & 1 & 0 & \cdots & 0 \\
    -x_2 & 0 & 1 & \ddots & \vdots \\
    \vdots & \vdots & \ddots & \ddots & 0 \\
    -x_{r-1} & 0 & \cdots & 0 & 1 \\
    -x_r & 0 & 0 & \cdots & 0
  \end{pmatrix}
  \cdot
  \begin{pmatrix}
    h_1^{(n)} \\
    h_2^{(n)} \\
    \vdots \\
    h_{r-1}^{(n)} \\
    h_r^{(n)} 
  \end{pmatrix}
  \label{eqn:RecursionForGrassmannRelationsInCohomology}
\end{equation}
Comparing this with the construction of the infinite Grassmannian 
$\Grass(r,\infty) = \bigcup_{n=r}^{\infty} \Grass(r,n)$ we 
see the following: Since the cohomology ring of $\Grass(r,n)$ is generated 
by the Chern classes $x_1,\dots,x_r$ of the tautological bundle 
for \textit{every} $n$ and the tautological bundle on $\Grass(r,n+1)$ restricts 
to the one on $\Grass(r,n)$, the pullback in cohomology for the 
inclusion $\Grass(r,n) \hookrightarrow \Grass(r,n+1)$ is given by 
\[
  \ZZ[x_1,\dots,x_r]/J_{r,n+1} \to \ZZ[x_1,\dots,x_r]/J_{r,n}
\]
where the containment of ideals $J_{r,n+1} \subset J_{r,n}$ is 
confirmed by the recurrence relation (\ref{eqn:RecursionForGrassmannRelationsInCohomology}) 
above.

\subsection{Cohomology of the matrix orbits}

This section will entirely consist of the proof of the following:

\begin{proposition}
  \label{prp:CohomologyOfTheOrbitVarieties}
  Let $0<r\leq  m\leq n$ be integers. The cohomology of the 
  orbit variety $O_{m,n}^{r}$ is graded isomorphic to 
  the Koszul algebra 
  \[
    H^\bullet\left( O_{m,n}^{r}\right) \cong 
    \bigwedge \left( R_m^{r} \cdot \eta_1 \oplus \dots \oplus 
    R_{m}^{r} \cdot \eta_r \right)
  \]
  over the ring $R_{m}^{r} = H^\bullet( \Grass(r,m))$ 
  with each $\eta_j$ a free generator of degree $2n-2j+1$.
\end{proposition}

We use the Cartan model (\ref{eqn:BorelComplex}) for  
$O_{m,n}^{r}$. The 
group in question is $U_{m,n} = U(m) \times U(n)$ with the subgroup 
\begin{eqnarray*}
  U_{m,n}^{r} &:=& \Stab(\one_{m,n}^{r}) \\
  &=& \left\{ (S \oplus P, \,
  S \oplus Q ) \in U(m)\times U(n)\right\}\\
  &\cong& U(r) \times 
  U(m-r) \times U(n-r).
\end{eqnarray*}
As a maximal torus of this subgroup we choose pairs of diagonal block matrices
\[
  \left( \diag(\underline \lambda) \oplus \diag(\underline \mu), \quad
  \diag(\underline \lambda) \oplus \diag(\underline \nu) 
  \right)
\]
with all nontrivial entries 
$\lambda_1,\dots,\lambda_{r},\mu_1,\dots,\mu_{m-r}, \nu_1,\dots,\nu_{n-r} \in U(1)$.
This is contained in the maximal torus $T$ of $U_{m,n}$ in the obvious way. 

We will write $S_{m,n}^{r}$ for the ring $H^\bullet(BU_{m,n}^{r})$. 
If we let $\{\alpha_j\}_{j=1}^{r}$ be the set of Chern roots associated to the 
subgroup $U(r)$, 
$\{\beta_j\}_{j=1}^{m-r}$ the ones for $U(m-r)$ and $\{\gamma_j\}_{j=1}^{n-r}$
those for $U(n-r)$, then similar to the case of flag 
varieties, the ring $S_{m,n}^{r}$ is the invariant 
subring of $\ZZ[\underline \alpha, \underline \beta, \underline \gamma]$ 
by the action of the group 
\[
  \mathfrak S_{m,n}^{r}:= 
  \mathfrak S_{r} 
  \times \mathfrak S_{m-r} \times \mathfrak S_{n-r}
\]
acting by the permutation of the distinct sets of variables. 
We set 
\begin{eqnarray*}
  x_j &=& \sigma_j^r(\alpha_1,\dots,\alpha_{r}) \quad 
  \textnormal{ for } j=1,\dots,r,\\
  y_k &=& \sigma_k^{m-r}(\beta_1,\dots,\beta_{m-r}) \quad 
  \textnormal{ for } k=1,\dots,m-r, \\
  z_l &=&  \sigma_l^{n-r}(\gamma_1,\dots,\gamma_{n-r}) \quad 
  \textnormal{ for } l=1,\dots,n-r 
\end{eqnarray*}
so that $S_{m,n}^{i_1,\dots,i_p} \cong \ZZ[\underline c, \underline y, \underline z]$ 
is a free polynomial ring containing $\ZZ[\underline x,\underline y, \underline z]$ 
as another free polynomial subring.

Since the group $U_{m,n} = U(m) \times U(n)$ is a product, also the exterior algebra 
in (\ref{eqn:BorelComplex}) takes the form of a product: 
\[
  \bigwedge (F_m \oplus F_n) \cong \left( \bigwedge F_m\right) \otimes_\ZZ 
  \left( \bigwedge F_n \right).
\]
We let $\{\varepsilon_j\}_{j=1}^m$ be the generators of $F_m$ and 
$\{\varepsilon'_k\}_{k=1}^n$ those 
of $F_n$. With the notation above it is easy to see from the 
inclusion of maximal tori that the differential $D$ of 
(\ref{eqn:BorelComplex}) takes these generators to the elements 
\begin{eqnarray*}
  D(1 \otimes \varepsilon_j) &=&  \sum_{r+s=j} x_r \cdot y_s,\\
  D(1\otimes \varepsilon'_k) &=&  \sum_{r+s=k} x_r \cdot z_s
\end{eqnarray*}
in $S_{m,n}^{i_1,\dots,i_p}$ and, as already discussed 
in Remark \ref{rem:CohomologyOfFlagVarietyAsChernClasses}, 
these are precisely the relations between the Chern 
classes of the tautological sub- and quotient bundles. 

We now consider Koszul complexes on $S_{m,n}^{i_1,\dots,i_p}$
associated to various 
subsets of the generators $D(1\otimes \varepsilon_j)$ 
and $D(1\otimes \varepsilon'_k)$. 
As discussed earlier in Remark \ref{rem:CohomologyOfFlagVarietyAsChernClasses}
the elements 
\[
  D(1\otimes \varepsilon_1), \dots
  D(1\otimes \varepsilon_{m-r}), 
  D(1\otimes \varepsilon'_1), \dots
  D(1\otimes \varepsilon'_{n-r}) 
\]
form a regular sequence on $\ZZ[\underline x, \underline y, \underline z]$ 
that can be used to eliminate the variables $\underline y$ and 
$\underline z$. Using the reduction lemma for 
the homology of Koszul complexes, Lemma \ref{lem:KoszulComplexAndRegularElements}, 
this reduces the problem to a Koszul complex on the ring 
\[
  \frac{\ZZ[x_1,\dots,x_r,y_1,\dots,y_{m-r},z_1,\dots,z_{n-r}]}
  {\langle D(1\otimes \varepsilon_1), \dots, D(1\otimes \varepsilon_{m-r}),
  D(1\otimes \varepsilon'_1), \dots, 
  D(1\otimes \varepsilon'_{n-r}) \rangle}
  \cong
  \ZZ[x_1,\dots,x_r] 
\]
In this quotient, the next $r$ relations reduce to 
\[
  \overline{D(1\otimes \varepsilon_{m-r+1})}
  = h^{(m)}_{1},\quad \dots, \quad 
  \overline{D(1\otimes \varepsilon_{m})}
  = h^{(m)}_r
\]
and they form another regular sequence with successive quotient 
$H^\bullet(\Grass(r,n))= \ZZ[x_1,\dots,x_r]/J_{r,n}$ as in 
Remark \ref{rem:CohomologyOfFlagVarietyAsChernClasses}. Consequently, 
the last $r$ relations 
\[
  \overline{D(1\otimes \varepsilon'_{n-r+1})}
  = h^{(n)}_1, \quad \dots,\quad 
  \overline{D(1\otimes \varepsilon'_{n})}
  = h^{(n)}_{r}
\]
all reduce to zero in $H^{\bullet}(\Grass(r,m))$ due to 
(\ref{eqn:RecursionForGrassmannRelationsInCohomology}) since by 
assumption $m \leq n$. 
In terms of the identification of the homology of Koszul complexes 
from Lemma \ref{lem:KoszulComplexAndRegularElements} this means 
\begin{eqnarray*}
  & & H\left( \ZZ[x,y,z] \otimes_{\ZZ} \bigwedge \left( F_m \oplus F_n \right) \right)\\
  &\cong& H\left( \ZZ[x] \otimes_{\ZZ} \bigwedge \left( \ZZ \varepsilon_{m-r+1} \oplus \dots 
  \oplus \ZZ \varepsilon_m \oplus 
  \ZZ \varepsilon'_{n-r+1} \oplus \dots \oplus \ZZ \varepsilon'_n \right) \right)\\
  & \cong & 
  H\left( \ZZ[x]/J_{m,r} \otimes_{\ZZ} \bigwedge \left( 
  \ZZ \varepsilon_{n-r+1}' \oplus \dots \oplus \ZZ \varepsilon'_n\right)\right)\\
  & \cong & 
  \ZZ[x]/J_{m,r} \otimes_{\ZZ} \bigwedge \left( 
  \ZZ \varepsilon_{n-r+1}' \oplus \dots \oplus \ZZ \varepsilon'_n\right)\\
\end{eqnarray*}
where the differentials of the complex in the second last line are all zero 
so that we may drop the homology functor $H(-)$. 
To finish the proof of Proposition 
\ref{prp:CohomologyOfTheOrbitVarieties} we now set $\eta_j$ to be equal to 
$1 \otimes \varepsilon'_{n-r+j}$ for $j=1,\dots,r$ in the last line.

\begin{remark}
  \label{rem:FiberBundleStructureIsCohomologicallyTrivial}
  We review the fiber bundle structure 
  \[
    \Stief(r,n) \hookrightarrow O_{m,n}^r \overset{\lambda_1}{\longrightarrow} \Grass(r,m)
  \]
  described in Lemma \ref{lem:FiberBundleStructureOfROrbitOverGrassmannian}. 
  Given the explicit descriptions of the cohomology of both the 
  Stiefel manifold and the Grassmannian, we may infer from the 
  proof of Proposition \ref{prp:CohomologyOfTheOrbitVarieties} that, 
  similar to the case the Hopf's theorem, the cohomology of 
  $O_{m,n}^r$ is isomorphic to that of a product 
  \[
    H^\bullet( O_{m,n}^r) \cong H^\bullet (\Stief(r,n)\times \Grass(r,m)),
  \]
  despite the fact that the structure of $O_{m,n}^r$ as a fiber bundle is in 
  general non-trivial. By construction, the cohomology classes $\eta_j$ restrict to the 
  generators of the cohomology of $\Stief(r,n)$ in every fiber.
  Yet it seems difficult to write down an explicit lift of these elements 
  to the original complex $\ZZ[x,y,z]\otimes \bigwedge \left( F_m \oplus F_n \right)$.
\end{remark}

\begin{remark}
  It would be nice to have a more geometric understanding of the elements 
  $\eta_j$ that might, for example, be derived from the tautological sequence
  \begin{equation}
    \xymatrix{
      0 \ar[r] & 
      \lambda_2^* Q_{2}^\vee \ar[r] & 
      \OO^n \ar[r]^\varphi & 
      \OO^m \ar[r] & 
      \lambda_1^* Q_1 \ar[r] & 
      0
    }
    \label{eqn:TautologicalSequenceOnTheOrbitVariety}
  \end{equation}
  on $O_{m,n}^{r}$ where the $Q_i$ denote the tautological quotient bundles on 
  respective Grassmannians for the projections $\lambda_1$ and 
  $\lambda_2$ as in Lemma \ref{lem:FiberBundleStructureOfROrbitOverGrassmannian}
  and $\varphi$ denotes the \textit{tautological section} 
  \[
    O_{m,n}^{r} \to \Hom( \OO^n,\OO^m ), \quad 
    \varphi \mapsto \varphi
  \]
  on $O_{m,n}^{r}$, identifying the 
  subbundles
  $\lambda_2^* S_2^\vee$ with $\lambda_1^* S_1$.
\end{remark}

\section{Betti numbers of smooth links} 
\label{sec:BettiNumbersOfSmoothLinks}

A hyperplane $(D_i,0) \subset (\CC^{m \times n},0)$ of codimension $i$ 
in general position will 
intersect $M_{m,n}^{r+1}$ in a non-trivial way 
whenever $i \leq \dim M_{m,n}^{r+1}$. 
This intersection $(M_{m,n}^{r+1} \cap D_i,0)$ will have isolated 
singularity as long as $D_i$ meets the singular locus 
$M_{m,n}^{r}$ of $M_{m,n}^{r+1}$ only at the origin, i.e. 
$i\geq \dim M_{m,n}^{r}$. 
The real and complex links of codimension $i$ of 
$(M_{m,n}^{r+1},0)$ will therefore be smooth within the range 
\[
  (m+n)(r-1)-(r-1)^2 \leq i < (m+n)r-r^2.
\]
In this section we will be concerned with the cohomology 
of $\mathcal K^i(M_{m,n}^{r+1},0)$ and $\mathcal L^i(M_{m,n}^{r+1},0)$ 
for $i$ in this range. 
More precisely, we will show for that for $i$ in the above range 
\begin{eqnarray}
  H^k(\Grass(r,m)) &\cong& H^k(\mathcal K^i(M_{m,n}^{r+1},0)) \quad 
  \textnormal{ for } \quad k<d(i),
  \label{eqn:CohomologyOfSmoothRealLinksBelowMiddle}
  \\
  H^k(\Grass(r,m)) &\cong& 
  H^k(\mathcal L^i(M_{m,n}^{r+1},0)) \quad 
  \textnormal{ for } \quad k<d(i), 
  \label{eqn:CohomologyOfSmoothComplexLinksBelowMiddle}
  \\
  H^{d(i)}(\Grass(r,m)) &\subset& 
  H^{d(i)}(\mathcal K^i(M_{m,n}^{r+1},0)) \subset 
  H^{d(i)}(\mathcal L^i(M_{m,n}^{r+1},0))
  \label{eqn:CohomologyOfSmoothComplexLinksMiddle}
\end{eqnarray}
where $d(i) = \dim_\CC \mathcal L^i(M_{m,n}^{r+1},0) = (m+n)r-r^2-i-1$. 
The left hand sides all agree with the cohomology of $V_{m,n}^r$ 
in this range for $k$ 
and the above maps are given by the pullback in cohomology 
for the natural inclusions of the 
real and complex links into that stratum. 

For the complex links the middle Betti number can then be computed 
from the polar multiplicities using Formula (\ref{eqn:FormulaForEulerCharacteristic}). 
In case $r=1$ this gives a complete description of the classical real and complex 
links of $(M_{m,n}^2,0)$, see Remark \ref{rem:MiddleCohomologyForRank1}.

However, besides the case $i=0$ and $r=1$, we are unable to determine 
whether or not $H^{d(i)}(\mathcal L^i(M_{m,n}^{r+1},0))$ has torsion. 
Also, the two middle cohomology groups of $\mathcal K^i(M_{m,n}^{r+1},0)$ can not 
be computed by our methods for $i>0$.

\subsection{The variation sequence on stratified spaces}
\label{sec:VariationSequence}

Throughout this section, let $(X,0) \subset (\CC^q,0)$ 
be an equidimensional reduced complex analytic 
germ of complex dimension $d$, endowed with a complex analytic 
Whitney stratification $\{ V^\alpha \}_{\alpha \in A}$.
We will denote by $f \colon (\CC^n,0) \to (\CC,0)$ the germ of a holomorphic 
function with an isolated singularity on $(X,0)$ in the stratified sense. 
For a Milnor ball $B_\varepsilon$ of sufficiently small radius $\varepsilon > 0$ 
and representatives $X$ and $f$ of the space and the function we let 
\[
  \partial X := \partial B_\varepsilon \cap X
\]
be the real link of $(X,0)$ and 
\[
  M = X \cap B_\varepsilon \cap f^{-1}(\{\delta\})
\]
the Milnor fiber of $f$ for some $\varepsilon \gg \delta > 0$. 
Denote the regular part of $X$ by $X_{\reg}$. Then clearly by 
construction of $\partial X$ and $M$ 
\[
  \partial X_{\reg} := \partial B_\varepsilon \cap X_{\reg}, \quad 
  M_{\reg} := M \cap X_{\reg} 
\]
are the regular loci of $\partial X$ and $M$, respectively.

\begin{lemma}
  \label{lem:VariationOnRegularLocus}
  The natural maps in cohomology 
  \[
    H^k(\partial X_{\reg}) 
    \to
    H^{k}(M_{\reg}) 
  \]
  are injective for $k=d-1$ and an isomorphism for $k<d-1$. 
\end{lemma}

\begin{proof}
  It is well known that we may ``inflate'' the Milnor-L\^e fibration for 
  $f$ and identify its total space with an open subset of $\partial X$ itself:
  \[
    \xymatrix{
      X \cap B_\varepsilon \cap f^{-1}\left( \partial D_\delta \right) \ar[r]^\cong 
      \ar[d]^{f} & 
      \partial X \setminus f^{-1}(D_\delta) \ar[d]^{\arg f} \\
      \partial D_\delta \ar[r]^{\cong} & 
      S^1.
    }
  \]
  See for example \cite[Part II, Chapter 2.A, Proposition 2.A.3]{GoreskyMacPherson88}, 
  or \cite[Part II, Chapter 2.4]{GoreskyMacPherson88} for the particular case 
  where $f$ is linear.

  Since $f$ has isolated singularity on $(X,0)$ it is a stratified 
  submersion on $X$ near the boundary $K := \partial X \cap f^{-1}(\{0\})$ 
  of the central fiber. We may therefore identify 
  \[
    \partial X \cap f^{-1}(D_\delta) \cong K \times D_\delta
  \]
  and extend the fibration by $\arg f$ over $K \times D_\delta^*$ 
  to the whole complement of $K \subset \partial X$.
  
  The assertion is now a consequence of the variation sequence for the regular 
  loci.
  It is evident from Thom's first isotopy lemma that 
  \[
    \arg f \colon \partial X \setminus K \to S^1 
  \]
  respects 
  the stratification of $M \subset \partial X \setminus K$ induced from $X$.
  Therefore, the common chain of 
  isomorphisms for the variation sequence 
  \begin{eqnarray*}
    H^k\left( \partial X_{\reg}, M_{\reg} \right) & \cong & 
    H^k\left( \partial X_{\reg}, 
    M_{\reg} \cup \left( \partial X_\reg \cap f^{-1}( \overline D_\delta ) \right) \right) \\
    &\cong& H^{k}\left( M_{\reg} \times [0,1], 
    K_{\reg} \times [0,1] \cup M_{\reg} \times \{0,1\} \right) \\
    &\cong& H^{k}\left( (M_{\reg}, K_{\reg}) \times ([0,1], \partial [0,1] ) \right) \\
    &\cong& H^{k-1}( M_{\reg}, K_{\reg} )
  \end{eqnarray*}
  restricts to the regular loci for every $k$. 
  Here, we deliberately identified $K_{\reg}$ with the boundary part 
  $\partial B_\varepsilon \cap M_{\reg}$ of the regular locus of $M$. 
  In the last step we used the 
  K\"unneth formula for pairs of spaces, see e.g. \cite[Theorem 3.21]{Hatcher02}.
  For a more thorough treatment see also 
  \cite[Part II, Chapter 2.5]{GoreskyMacPherson88}.
  
  With these identifications, the long exact sequence of the pair 
  $( \partial X_\reg, M_\reg)$ reads 
  \[
    \dots 
    \to 
    H^{k-1}( M_\reg, K_\reg ) \to 
    H^k( \partial X_\reg ) \to 
    H^k( M_\reg ) \to 
    H^{k}( M_\reg, K_\reg ) \to 
    \dots
  \]
  The assertion now follows from the fact that $H^k( M_\reg, K_\reg ) = 0$ 
  for all $k < d-1 = \dim_\CC M_\reg$ which is due to complex stratified Morse theory 
  for non-proper Morse functions, see \cite[Introduction, Chapter 2.2]{GoreskyMacPherson88}.
\end{proof}

\begin{remark}
  The cited theorem \cite[Introduction, Chapter 2.2]{GoreskyMacPherson88} is a statement on 
  the relative homology for the smooth locus of a projective algebraic variety 
  and its intersection with a generic hyperplane. The statement can be generalized 
  to germs of complex analytic sets and their intersections with a suitable
  ball, cf. \cite[Introduction, Chapter 2.4]{GoreskyMacPherson88}. Then one can 
  use complex stratified Morse theory for non-proper Morse functions to study 
  the connectivity of the pair $(M_\reg, K_\reg)$ via a Morsification 
  $\rho$ of the squared distance function to the origin. 
  The key observation is that for the regular locus $M_\reg$ 
  the local Morse datum at a critical point is always $(d-2)$-connected. 
  The last statement can easily be derived by induction on dimension from 
  \cite[Part II, Chapter 3.3, Corollary 1]{GoreskyMacPherson88}.
\end{remark}

Now suppose we are given $(X,0) \subset (\CC^q,0)$ as above and 
an admissible sequence of linear forms $l_1,\dots,l_d$ in the 
sense of Definition \ref{def:AdmissibleSequenceOfLinearForms} 
and Corollary \ref{cor:ParticularChoiceOfAdmissibleLinearForms}.
Then we can apply Lemma \ref{lem:VariationOnRegularLocus} inductively 
to $f = l_{i+1}$ on $(X \cap D_i,0)$:

\begin{proposition} Let $(X,0) \subset (\CC^q,0)$ be an equidimensional 
  reduced complex analytic germ of dimension $d$, and 
  $l_1, l_2, \dots, l_d$ an admissible sequence of linear forms 
  as in Corollary \ref{cor:ParticularChoiceOfAdmissibleLinearForms}. 
  Let $\mathcal K^k_\reg$ and $\mathcal L^k_\reg$ be the regular loci 
  of the real and complex links of codimension $k$ of $(X,0)$. 
  Then for arbitrary $0\leq k< d$ the natural maps 
  \[
    H^i(\mathcal K^0_\reg ) \to H^i(\mathcal K^k_\reg) \to 
    H^i(\mathcal L^k_{\reg}) 
  \]
  are isomorphisms for every $i<d-k-1 = \dim \mathcal L^k$ 
  and injective for $i=d-k-1$. 
  \label{prp:IsomorphismLowDimensionalHomology}
\end{proposition}

\begin{proof}
  We proceed by induction on the codimension $k$ of the complex links. 
  For $k=0$ this follows 
  directly from Lemma \ref{lem:VariationOnRegularLocus}. 

  Now suppose the statement holds up to codimension $k-1$. 
  As already discussed in the proof of Lemma \ref{lem:VariationOnRegularLocus}, 
  the real link $\mathcal K^k$ can be identified with the boundary 
  of $\mathcal L^{k-1}$. 
  But the pair $(\mathcal L^{k-1}_\reg, \mathcal K^k_\reg)$ 
  is cohomologically $d-k-1$-connected due to the LHT and the long exact sequence yields 
  \[
    0 \to H^i(\mathcal L^{k-1}_\reg) \overset{\cong}{\longrightarrow} H^i( \mathcal K^k_\reg) \to 0 
  \]
  for $i<d-k-1$ and 
  \[
    0 \to H^{d-k-1}(\mathcal L^{k-1}_\reg) \to H^{d-k-1}(\mathcal K^k_\reg) \to 
    H^{d-k}(\mathcal L^{k-1}_\reg, \mathcal K^k_\reg) \to \cdots
  \]
  for the middle part. Since by our induction hypothesis the cohomology groups 
  $H^{i}(\mathcal L^{k-1}_\reg) \cong H^{i}(\mathcal K^0_\reg)$ are all isomorphic 
  for $i<d-k$, the first part of the assertion on $H^i(\mathcal K^0_\reg) \to H^i(\mathcal K^k_\reg)$
  follows.

  For the second part on $H^i(\mathcal K^0_\reg) \to H^i(\mathcal L^k_\reg)$ consider 
  \[
    l_{k+1} \colon (X \cap D_k,0) \to (\CC,0).
  \]
  By construction the Milnor fiber of this function is $\mathcal L^{k}$ 
  while $\mathcal K^{k}$ is the real link of $(X \cap D_k,0)$. All the relevant 
  cohomology groups of $\mathcal K^k$ have already been determined by 
  our previous considerations so that the remaining statements follow 
  readily from Lemma \ref{lem:VariationOnRegularLocus} applied to $f = l_{k+1}$.

%
\end{proof}

\subsection{Proof of Formulas (\ref{eqn:CohomologyOfSmoothRealLinksBelowMiddle}) -- (\ref{eqn:CohomologyOfSmoothComplexLinksMiddle})}
As remarked earlier, the relevant range for the codimension $i$ 
for the real and complex links 
$\mathcal K^i(M_{m,n}^{r+1},0)$ and 
$\mathcal L^i(M_{m,n}^{r+1},0)$ 
of the generic determinantal variety $(M_{m,n}^{r+1},0)$ is 
\[
  \dim M_{m,n}^{r} = (m+n)(r-1) - (r-1)^2 \leq i < (m+n)r-r^2 = \dim M_{m,n}^{r+1}.
\]
We may assume $m \leq n$. Then for these values of $i$ 
the dimension of the complex links does not exceed 
\[
  d(i):= \dim_\CC \mathcal L^i(M_{m,n}^{r+1},0) 
  = (m+n)r - r^2 - i - 1 < m+n-2r+1 \leq 2n -2r +1.
\]
Recall that due to the homogeneity of the singularity, 
the cohomology of the regular part of the classical 
real link of $(M_{m,n}^{r+1},0)$ is given by 
\begin{eqnarray*}
  H^\bullet(\mathcal K^0_\reg(M_{m,n}^{r+1},0)) 
  & \cong & H^\bullet( V_{m,n}^r ) \cong  H^\bullet( O_{m,n}^r ) \\
  & \cong & 
    \bigwedge \left( R_m^{r} \cdot \eta_1 \oplus \dots \oplus 
    R_{m}^{r} \cdot \eta_r \right)
\end{eqnarray*}
according to Proposition \ref{prp:CohomologyOfTheOrbitVarieties}. 
Now note that the bound on $d(i)$ 
assures that each and every generator $\eta_{j} \in H^{2n-2j+1}(O_{m,n}^r)$ 
is taken to a vanishing cohomology group of a \textit{smooth} complex 
link $\mathcal L^i(M_{m,n}^{r+1},0)$. 
The formulas (\ref{eqn:CohomologyOfSmoothRealLinksBelowMiddle}), 
(\ref{eqn:CohomologyOfSmoothComplexLinksBelowMiddle}), and 
(\ref{eqn:CohomologyOfSmoothComplexLinksMiddle}) now follow 
directly from Proposition \ref{prp:IsomorphismLowDimensionalHomology}.

\subsection{The middle cohomology groups}

The results of the previous section allow to determine 
the cohomology \textit{below the middle degrees} 
of the real and complex links $\mathcal K^i$ and $\mathcal L^i$ of codimension $i$ 
for a purely $d$-dimensional germ $(X,0)$, provided $i\geq \dim X_\sing$ 
so that $(X \cap D_i,0)$ has isolated singularity and the links are smooth. 
We shall now discuss how to obtain information on the remaining part 
of the cohomology and to which extend this is possible. 

When $i\geq \dim X_\sing$ 
the complex link $\mathcal L^i = \mathcal L^i_\reg$ is Stein and 
the higher cohomology groups vanish due to the LHT. We can use 
Lefschetz duality (see e.g. \cite[Theorem 3.43]{Hatcher02}) to 
identify homology with cohomology:
\[
  H^k(\mathcal L^i,\partial \mathcal L^i) \cong H_{2(d-i-1)-k}(\mathcal L^i) 
  \quad \textnormal{and} \quad 
  H^k(\mathcal L^i) \cong H_{2(d-i-1)-k}(\mathcal L^i, \partial \mathcal L^i).
\]
Note that the middle homology group $H_{d-i-1}(\mathcal L^i)$ is known to 
always be free; this is the case even for arbitrary smooth Stein manifolds. 
The real link 
$\mathcal K^i$ is an oriented smooth compact manifold of real dimension 
$2d-2i-1$ and we have Poincar\'e duality (cf. \cite[Theorem 3.30]{Hatcher02}):
\[
  H^k( \mathcal K^i) \cong H_{2d-2i-1-k}( \mathcal K^i)
\]
for all $k$. All these cohomology groups sit in the classical variation sequence 
with its middle part being of particular interest:
\[
  0 \to 
  H^{d-i-1}( \mathcal K^i ) \to 
  H^{d-i-1}( \mathcal L^i )  
  \overset{\mathrm{VAR}}{\longrightarrow} 
  H^{d-i-1}( \mathcal L^i, \partial \mathcal L^i) \to
  H^{d-i}( \mathcal K^{i} )
  \to 0.
\]

\medskip
We return to the particular case of the generic determinantal singularity 
$(X,0) =(M_{m,n}^{r+1},0) \subset (\CC^{m\times n},0)$.
For the smooth complex links $\mathcal L^i = \mathcal L^i(M_{m,n}^{r+1},0)$ we can use the knowledge of 
the Euler characteristic from the polar multiplicities 
in Section \ref{sec:EulerCharacteristicComplexLinks}, Formula 
(\ref{eqn:FormulaForEulerCharacteristic}) to also determine 
the rank of the middle cohomology group: 
\begin{equation}
  b_{d-i-1}(\mathcal L^i) = 
  \underbrace{\left( \sum_{j=i}^{d}(-1)^{j} m_0\left( P_{d-j}(X,0) \right) \right)}_{ 
    (-1)^{d-i-1} \chi( \mathcal L^i ) 
  }
  + \underbrace{\left(\sum_{j=0}^{d-i-2} (-1)^{d-i-j} b_j(O_{m,n}^r) \right).}_{
    -(-1)^{d-i-1} \sum_{j=0}^{d-i-2} (-1)^j b_j(\mathcal L^i)
  }
  \label{eqn:MiddleBettiNumberComplexLink}
\end{equation}
Switching from integer to rational coefficients, this allows us to fully compute 
the rational cohomology, 
but unfortunately this method does \textit{not} allow to detect torsion in the middle 
cohomology group with integer coefficients.
However, so far no example of a smooth complex link of a generic determinantal variety is known 
for which the middle cohomology group 
$H^{d-i-1}(\mathcal L^i(M_{m,n}^{r+1},0))$ does have torsion. 

\medskip
Being an oriented smooth compact manifold of odd real dimension, $\chi(\mathcal K^i)$ 
is always zero and therefore computation of the Euler characteristic does not 
help in this case; not even the full \textit{rational}
cohomology of $\mathcal K^i$ can be computed 
by our methods. It should be noted, however, that there are examples 
for which torsion appears: 

\begin{example}
  Consider the singularity $(Y,0) := (M_{2,3}^2\cap D_2,0) \subset (\CC^{2\times 3},0)$ 
  for a generic 
  plane $D_2$ of 
  codimension $2$.
  This is an isolated normal surface singularity, 
  the simplest one of the ``rational triple points'' discussed by Tjurina \cite{Tjurina68}. 
  It was shown in \cite{Zach18BouquetDecomp} that the 
  so-called Tjurina modification $\pi \colon Y' \to Y$ of $(Y,0)$ is smooth and 
  hence a resolution of singularities for $(Y,0)$. The space $Y'$ is 
  isomorphic to the total space of the bundle 
  \[
    Y' \cong | \OO_{\PP^1}(-3) |
  \]
  and hence the complex link 
  $\mathcal K^2 = \mathcal K^2(M_{2,3}^2,0)$ can be identified with the sphere bundle 
  of $\OO_{\PP^1}(-3)$. Then a part of the Euler sequence for this bundle reads 
  \[
    \dots H^0(\PP^1) \overset{\cup e}{\longrightarrow} H^2(\PP^1) \to 
    H^2(\mathcal K^2) \to 0.
  \]
  But with the canonical generators for $H^\bullet(\PP^1)$, the cup product with the 
  Euler class $e$ is simply multiplication by $-3$ and we find that 
  \begin{equation}
    H^2(\mathcal K^2(M_{2,3}^2,0)) \cong \ZZ/3\ZZ.
    \label{eqn:TorsionExampleForTheRealLink}
  \end{equation}
  The complex link $\mathcal L^2 = \mathcal L^2(M_{2,3}^2,0)$, i.e. the ``Milnor fiber'' 
  in the variation sequence, is known to have middle cohomology 
  $H^1(\mathcal L^2) \cong \ZZ^2$; see e.g. \cite{Zach18BouquetDecomp}. Since 
  $H^1(\mathcal L^2, \partial \mathcal L^2) \cong H_1(\mathcal L^2)\cong \ZZ^2$ is also free 
  and the map $\mathrm{VAR}$ necessarily has rank $2$, it follows that 
  \[
    H^1(\mathcal K^2(M_{2,3}^2,0)) = 0.
  \]
\end{example}

\medskip
To conclude this section, we mention one last case that might be of particular interest:

\begin{remark}
  \label{rem:MiddleCohomologyForRank1}
  For $r=1$ the generic determinantal varieties $(M_{m,n}^2,0)$ all have isolated singularity 
  and therefore the classical real and complex links $\mathcal K^0 = \mathcal K^0(M_{m,n}^2,0)$ 
  and $\mathcal L^0 = \mathcal L^0(M_{m,n}^2,0)$ are already 
  smooth. In this case we know all the cohomology groups of $\mathcal K^0$  
  \[
    H^\bullet( \mathcal K^0(M_{m,n}^2,0)) \cong H^\bullet( \PP^{m-1} \times S^{2n-1} )
    \cong H^\bullet( \PP^{m-1}) \otimes_\ZZ \bigwedge \ZZ \varepsilon_{2n-1}
  \]
  which is free Abelian in all degrees.
  Using the variation sequence we infer that the middle cohomology of $\mathcal L^0$ 
  can not have torsion: By Poincar\'e duality also the relative cohomology group to the 
  left is isomorphic to the middle homology group of $\mathcal L^0$. Now the latter 
  is known to always be free Abelian for Stein manifolds. 
  Therefore we have 
  \begin{equation}
    H^k(\mathcal L^0(M_{m,n}^2,0)) \cong 
    \begin{cases}
      \ZZ & \textnormal{ if $k\leq (m-r)r$ is even } \\
      0 & \textnormal{ otherwise.}
    \end{cases}
    \label{eqn:CohomologyClassicalComplexLinkRankOne}
  \end{equation}
\end{remark}

We would like to mention one particular consequence of the previous remark. 
It is well known that on a locally complete intersection $X$ the constant 
sheaf $\ZZ_X[\dim X]$ is perverse for the middle perversity; see e.g. 
\cite[Theorem 5.1.20]{Dimca04}. Most of the generic determinantal varieties 
$M_{m,n}^s$ are not complete intersections and we can now show that this 
algebraic property is also reflected in their topology:

\begin{corollary}
  For  $1<s\leq m < n$, i.e. for non-square matrices, 
  the constant sheaf $\ZZ_{M_{m,n}^{s}}[\dim M_{m,n}^s]$ 
  on the generic determinantal variety, shifted by its dimension, 
  is never a perverse sheaf for the middle perversity.
\end{corollary}

\begin{proof}
  For these values of $s,m$, and $n$ the rank stratification is the minimal 
  Whitney stratification of $M_{m,n}^{s}$. Now the constant sheaf can only be perverse
  (for the middle perversity)
  if the variety has a certain \textit{rectified homological depth}, 
  meaning that along every stratum $V_{m,n}^r \subset M_{m,n}^s$ 
  the real link $\mathcal K(M_{m,n}^s,V_{m,n}^r)$ along $V_{m,n}^r$ satisfies
  \[
    H_k(\mathcal K(M_{m,n}^s,V_{m,n}^r)) = 0 \textnormal{ for all } 
    k < \codim(V_{m,n}^r,M_{m,n}^s)-1.
  \]
  See \cite[Section 1.1.3]{HammLe91}, \cite[Theorem 1.4]{HammLe91} for a discussion 
  on rectified homological depth and 
  e.g. the proof of \cite[Theorem 5.1.20]{Dimca04} for its relation to perversity.
  For $r = s-2$ the codimension on the right hand side is 
  \[
    c = (m-s+2)(n-s+2) - (m-s+1)(n-s+1) = m+n -2s + 3 \geq 4
  \]
  for all values $1<s\leq m < n$. But according to Remark 
  \ref{rem:MiddleCohomologyForRank1} we have 
  \[
    H^2( \mathcal K^0(M_{m,n}^s,V_{m,n}^{s-2})) \cong 
    H^2(\mathcal K^0(M_{m-s+2,n-s+2}^{2})) \cong H^2(\PP^{m-s+1}) \cong \ZZ
  \]
  so that the criterion for the perversity given above is violated.
\end{proof}

\section{Implications for the vanishing topology of smoothable IDS}

\subsection{Proof of Theorem \ref{thm:MainConsequence}}

We consider a $\mathrm{GL}$-miniversal unfolding 
\[
  \mathbf A \colon (\CC^p,0) \times (\CC^k,0) \to (\CC^{m\times n},0)
\]
of the defining matrix $A$ for $(X_A^s,0)$ on 
$k$ parameters $t_1,\dots,t_k$. 
See, for instance, \cite{FruehbisKruegerZach21} for the underlying 
notion of $\GL$-equivalence and the existence and construction 
of such miniversal unfoldings. The induced deformation 
\[
  \xymatrix{
    X_A^s \ar@{^{(}->}[r] \ar[d] &
    \mathcal X_{\mathbf A}^s \ar[d]^\pi \\
    \{ 0 \} \ar@{^{(}->}[r] & 
    \CC^k
  }
\]
given by the projection of 
$\mathcal X_{\mathbf A}^s = \mathbf A^{-1}(M_{m,n}^s) \subset \CC^p\times \CC^k$ 
to the parameter space $\CC^k$ is versal in the sense that it covers all 
possible determinantal deformations of $(X_A^s,0)$ coming from this 
particular choice of a determinantal structure. 

Let the \textit{discriminant} $(\Delta,0) \subset (\CC^k,0)$ be the set of parameters 
$t = (t_1,\dots,t_k)$ with singular fibers $X_{\mathbf A}^s(t)$. For 
a suitable representative of the unfolding $\mathbf A$ we may choose 
a Milnor ball $B_\varepsilon \subset \CC^p$ around the origin and 
a polydisc $0 \in D_\delta \subset \CC^k$ in the parameter space 
such that 
\[
  \pi \colon \mathcal X_{\mathbf A}^s \cap B_\varepsilon \cap (D_\delta \setminus \Delta) 
  \to (D_\delta \setminus \Delta )
\]
is a smooth fiber bundle of manifolds with boundary with fiber $M_A^s$, the smoothing 
of $(X_A^s,0)$. 

Since for $t \notin \Delta$ the map 
\[
  A_t = \mathbf A(-,t) \colon B_\varepsilon \to \CC^{m\times n}
\]
is transverse to the rank stratification implies that $A_t$ does not degenerate 
on the fiber $X_{\mathbf A}^s(t)$ in the sense that it has constant rank 
$r = s-1$ at all points $x \in X_{\mathbf A}^s(t)$. Therefore the cokernel 
\[
  \CC^n \overset{A_t}{\longrightarrow} \CC^m \to E \to 0
\]
presented by $A_t$ is a well defined vector bundle of rank $m-r$ on the 
fiber over $t$. 

Variation of $t$ in $D_\delta \setminus \Delta$ does not change the fiber 
$X_{\mathbf A}^s(t)$ up to diffeomorphism. In the same sense, it varies the 
vector bundle $E$ in a $C^\infty$ way on these fibers, but does not change its 
isomorphism class as a smooth complex vector bundle. 

Due to the versality of $\mathbf A$, 
the various deformations of $(X_A^s,0)$ and its smoothing $M_A^s$ 
in the proof of the bouquet 
decomposition (\ref{eqn:BouquetDecomposition}) in 
\cite{Zach18BouquetDecomp} can be realized using piecewise 
smooth paths in the parameter space $\CC^k$. In particular, 
this leads to a representative $X_{\mathbf A}^s(t)$ of $M_A^s$ which has 
an embedding of the complex link $\mathcal L^{mn - p -1}(M_{m,n}^s,0) \hookrightarrow 
X_{\mathbf A}(t)$. Now it follows from the results in Section 
\ref{sec:BettiNumbersOfSmoothLinks} that the truncated cohomology of the 
Grassmannian $H^{\leq d}( \Grass(r,m) )$ is generated by the algebra 
of Segre classes of the tautological quotient bundle. Since by construction 
the pullback of this bundle to both $X_{\mathbf A}(t)$ and 
$\mathcal L^{mn - p -1}(M_{m,n}^s,0)$ is the vector bundle $E$ in question, 
the result follows.

\subsection{Proof of Corollary \ref{cor:SecondCohomologyGroupSmoothingICMC2Threefold}}

Let $M \subset U \subset \CC^5$ be a smoothing of the isolated Cohen-Macaulay 
threefold $(X,0) \subset \CC^5$ on some open set $U$ on which a perturbation 
\[
  A_t \colon U \to \CC^{m \times (m+1)} 
\]
of the matrix $A$ is defined. By definition the canonical sheaf on 
$M$ is given by $\omega_M = \mathcal Ext^2_{\OO_U}(\OO_M,\OO_U)$ which can be computed from 
a free resolution of $\OO_M$ as an $\OO_U$-module. Such 
a free resolution is provided by the deformation of the resolution of 
$\OO_{X,0}$ from the Hilbert-Burch theorem. Now it is easy to see 
from the proof of Theorem \ref{thm:MainConsequence} that the line bundle 
$\omega_M$ is presented by the restriction of $A_t$ to $M_A^s \subset U$, 
i.e. it is the vector bundle $E$ in Theorem \ref{thm:MainConsequence} 
and the claim follows.

\section*{Acknowlegdements}
\noindent
The author wishes to thank
Duco van Straten for introducing him to the work of Borel on 
the cohomology of homogeneous spaces and Sam Hagh Shenas Noshari for 
further helpful conversations on the topic, Terence Gaffney for
discussions on polar varieties and --multiplicities, Xiping Zhang for 
an exchange on their computation, and James Damon for 
conversations on the ``characteristic cohomology'' of determinantal singularities.

\printbibliography
\end{document}